\documentclass[a4paper,12pt]{amsart}

\usepackage{
amsmath,amssymb,amsfonts,
bbm}
\usepackage{graphicx,
color}

\usepackage{geometry}
\usepackage{amsmath}

\usepackage{float}
\usepackage{caption}
\newtheorem{theorem}{Theorem}[section]
\newtheorem{lemma}[theorem]{Lemma}
\newtheorem{corollary}[theorem]{Corollary}
\newtheorem{proposition}[theorem]{Proposition}
\newtheorem{definition}[theorem]{Definition}

\theoremstyle{remark}
\newtheorem{remark}[theorem]{Remark}

\newtheorem*{remark*}{Remark}

\theoremstyle{definition}

\def\bR{\mathbb{R}}
\def\bC{\mathbb{C}}

\newcommand\cS{\mathcal{S}}

\def\Xint#1{\mathchoice
{\XXint\displaystyle\textstyle{#1}}%
{\XXint\textstyle\scriptstyle{#1}}%
{\XXint\scriptstyle\scriptscriptstyle{#1}}%
{\XXint\scriptscriptstyle\scriptscriptstyle{#1}}%
\!\int}
\def\XXint#1#2#3{{\setbox0=\hbox{$#1{#2#3}{\int}$}
\vcenter{\hbox{$#2#3$}}\kern-.5\wd0}}

\def\dashint{\Xint-}

\newcommand{\fint}{\dashint}

\numberwithin{equation}{section}

\begin{document}
	
\title[The Fefferman-Phong uncertainty principle]{The Fefferman-Phong uncertainty principle for representations of Lie groups and applications}
\author{Fabio Nicola}
\address{Dipartimento di Scienze Matematiche, Politecnico di Torino, Corso Duca degli Abruzzi 24, 10129 Torino, Italy.}
\email{fabio.nicola@polito.it}


\begin{abstract}
The Fefferman-Phong uncertainty principle is a deep result concerning the order of magnitude of the bottom of the spectrum of second order pseudodifferential operators with nonnegative (Weyl) symbol. We show that a similar uncertainty principle holds for discrete series representations of connected Lie groups, where the concentration of the matrix coefficients is measured in terms of weighted $L^p$ norms, with weights in the local Muckenhoupt class $A_{\infty,{\rm loc}}$ associated with a subRiemannian left-invariant metric and a relatively invariant measure. The proof relies on a certain connection with lower bounds for left-invariant subLaplacians.

As a consequence of this result (in the case of the Schr\"odinger representation of the reduced Heisenberg group), we provide an {\it explicit formula} for the order of magnitude of the bottom of the spectrum and also of the essential spectrum of semiclassical anti-Wick operators with nonnegative symbols of arbitrary order, hence providing the analog, for the anti-Wick quantization, of the above mentioned result by Fefferman and Phong. We consider symbols in standard symbol classes appearing in semiclassical analysis, and also in the Muckenhoupt classes (hence possibly non-smooth). 
 \end{abstract}
\maketitle
\section{Introduction and discussion of the main results}
This note is devoted to two related issues: a new uncertainty principle for discrete series representations of connected Lie groups, and an {\it explicit formula} for the order of magnitude of the bottom of the spectrum and of the essential spectrum of anti-Wick operators in $\bR^d$ with nonnegative symbols, analogous to that obtained by Fefferman and Phong \cite{fefferman_phong} for Weyl pseudodifferential operators. 
\subsection{The Fefferman-Phong uncertainty principle for representations}
Let ${{h}}>0$ and $\varphi(x)=\pi^{-d/4}e^{-\frac{1}{2}|x|^2}$, $x\in\bR^d$. Consider the so-called ``coherent states"
\begin{equation}\label{eq cs}
\varphi^{{{h}}}_{(x_0,\omega_0)}(x):={{h}}^{-d/4} e^{\frac{i}{{{h}}}\omega_0 x} \varphi({{h}}^{-1/2}(x-x_0)) \qquad x,x_0,\omega_0\in\bR^d.
\end{equation}
It is well known that the celebrated Heisenberg uncertainty principle can be rephrased in several ways as an inequality involving the functions $\varphi^{{{h}}}_{(x_0,\omega_0)}$. For example, one can show that the following inequality holds true, for $f\in L^2(\bR^d)$: 
\begin{equation}\label{eq uncertainty} 
(2\pi{{h}})^{-d}\int_{\bR^{2d}}(|x|^2+|\omega|^2)|\langle f,\varphi^{{{h}}}_{(x,\omega)}\rangle|^2\,dx\,d\omega\geq 2d{{h}}\|f\|^2_{L^2(\bR^d)}.
\end{equation}
Moreover equality occurs if and only if $f=c\varphi^{{h}}_{(0,0)}$ for some $c\in\bC$ (see e.g. \cite{daubechies,folland,grochenig} and Remark \ref{rem shubin} below).

The above expression $\langle f,\varphi^{{h}}_{(x,\omega)}\rangle$ can be regarded as a matrix coefficient of the Schr\"odinger representation of the reduced Heisenberg group (see e.g. \cite{wong} and Section \ref{sec heisenberg} below). This suggests that similar estimates can hold for other representations of Lie groups, where the concentration of the matrix coefficients is measured in terms of weighted $L^p$ norms.

The first part of this note is devoted precisely to this problem. We will prove that a similar lower bound in fact holds for every square-integrable irreducible unitary representation of a connected Lie group $G$ (under a certain mild assumption). The role of the weight  $|x|^2+|\omega|^2$ in \eqref{eq uncertainty} will be played by any function in the local Muckenhoupt class $A_{\infty,{\rm loc}}$ (see Definition \ref{def muck} below) associated with a subRiemannian left-invariant metric of $G$ and a relatively invariant measure. In particular, the group $G$ itself will play the role of the phase space (we refer to the recent contribution by Gr\"ochenig \cite{grochenig_nilpotent} for suggestive speculations about the possibility of regarding any {\it nilpotent} Lie group as a phase space).

To state our result precisely, let us introduce some notation and terminology. 
Let $G$ be a connected Lie group, and let $\mu$ be a right Haar measure of $G$. Let $\chi:G\to(0,+\infty)$ be a continuous character and let $\mu_\chi$ be the absolutely continuous measure that has density $\chi$ with respect to $\mu$. In particular, if $\chi$ is the modular function, $\mu_\chi$ is a left Haar measure, which will be also denoted by $\lambda$. 

Let $\pi:G\to\mathcal{U}(\mathcal{H})$ be a square-integrable strongly continuous irreducible unitary representation on a Hilbert space $\mathcal{H}$, i.e. a {\it discrete series representation} of $G$ . Hence, there exists $\phi\in\mathcal{H}$, with $\|\phi\|=1$, satisfying the ``admissibility condition"
\begin{equation}\label{eq admissibility} 
c_\phi:=\int_G |\langle \phi,\pi(x)\phi\rangle|^2\,d \lambda(x)<\infty. 
\end{equation}
For such a $\phi$, we define the corresponding {\it generalized wavelet transform} (also named ``voice transform" in the literature)
\begin{equation}\label{eq Vphi}
    V_\phi f(x)=\frac{1}{\sqrt{c_\phi}}\langle f,\pi(x)\phi\rangle,\qquad x\in G,\ f\in\mathcal{H}.
\end{equation}
It turns out that $V_\phi:\mathcal{H}\to L^2(G,\lambda)$ is an isometry, and that the function $V_\phi f$ is continuous on $G$, for every $f\in\mathcal{H}$. Many of the transforms arising in applied harmonic analysis (e.g. the short-time Fourier transform, the wavelet transform and the shearlet transform) fall within this general framework (see e.g. \cite{berge2022,feichtinger,fuhr,wong} and the references therein; see also Knapp's book \cite{knapp_book} for a comprehensive discussion of discrete series representations of semisimple Lie groups and their connection with complex analysis).

For $1\leq p<\infty$, we are interested in the following minimization problem: 
\begin{equation}\label{eq min pro} 
I(w):=\inf\Big\{\int_{G} w|V_\phi f|^p\,d\mu_\chi:\  f\in\mathcal{H}\textrm{ and }\|V_\phi f\|_{L^p(G,\mu_\chi)}=1\Big\},
\end{equation}
where $w:G\to[0,+\infty)$ is a Borel function.

 Let $X_1,\ldots,X_\ell $ be a family of left-invariant linearly independent vector fields
which satisfy H\"ormander’s condition, and denote by ${\rm d}_C(\cdot,\cdot)$ the associated left-invariant Carnot–Carath\'eodory distance (see Section \ref{sec lis}). Accordingly, we assume that the weight $w$ belongs to the local Muckenhoupt class $A_{\infty,{\rm loc}}$ associated with the metric measure space $(G,{\rm d}_C,\mu_\chi)$ (see \cite{bruno_calzi,grafakos,stein_book, stromberg} and Section \ref{sec prel} below). Let us briefly recall its definition. 
\begin{definition}[Local Muckenhoupt class $A_{\infty,{\rm loc}}$]\label{def muck}
Let $\alpha\in(0,1)$. We denote by $A_{\infty,{\rm loc}}^{\alpha}(G,{\rm d}_C,\mu_\chi)$ the set of all functions $w\in L^1_{\rm loc}(G,\mu_\chi)$, with $w>0$ almost everywhere, such that for all balls $B$ of radius $r\in(0,1]$ and all Borel subsets $E\subset B$ in the metric measure space $(G,{\rm d}_C,\mu_\chi)$, 
\begin{equation}\label{eq muck}
    \mu_\chi(E)\geq (1-\alpha) \mu_\chi(B) \Longrightarrow \int_E w\, d\mu_\chi\geq \alpha \int_B w\, d\mu_\chi.
\end{equation}
We also set $A_{\infty,{\rm loc}}(G,{\rm d}_C,\mu_\chi)=\cup_{0<\alpha<1}A_{\infty,{\rm loc}}^{\alpha}(G,{\rm d}_C,\mu_\chi)$.
\end{definition}
Here and in all the following by ``ball" we will always mean ``open ball".

Coming back to the problem \eqref{eq min pro}, notice that we have the trivial lower bound $I(w)\geq \inf w$. We will improve this estimate by replacing the infimum of $w$ by the infimum of the averages of $w$ over unit balls, that is 
\begin{equation}\label{eq iup}
    I_{\rm UP}(w):=\inf\Big\{\fint_B w\,d\mu_\chi:\ B\textrm{ is a ball of radius }1\Big\}.
\end{equation}

To this end, we will assume a mild horizontal regularity and decay condition on $V_\phi \phi$, namely that 
\begin{equation}\label{eq phi}
C_0(\phi):=\frac{1}{\sqrt{c_\phi}}\sum_{j=1}^\ell\int_G \chi^{-1/p} |X_j V_\phi\phi|d\lambda<\infty.
\end{equation}

\begin{theorem}\label{teomain}
    Let $1\leq p<\infty$, $\alpha\in(0,1)$, $\eta>0$. There exists a constant $C=C(\alpha,p,\eta)>0$ such that, for every $\phi\in\mathcal{H}$, with $\|\phi\|=1$, satisfying \eqref{eq admissibility} and \eqref{eq phi} with $C_0(\phi)\leq \eta$, and every weight $w\in A_{\infty,{\rm loc}}^{\alpha}(G,{\rm d}_C,\mu_\chi)$ we have 
    \[
      I(w)\geq C I_{\rm UP}(w).
    \]
\end{theorem}
The lower bound in Theorem \ref{teomain} is sharp in general, in the sense that in certain special cases one can prove that the upper bound $I(w)\leq C I_{\rm UP}(w)$ (for a different $C$) holds too, as we will see in a moment. 

Theorem \ref{teomain} is reminiscent of the Fefferman-Phong lower bound \cite{fefferman_phong,fefferman} for second order pseudodifferential operators with a nonnegative symbol $a(x,\omega)$ (see also Fefferman and S\'anchez-Calle \cite{fefferman1986}, Nagel, Stein and Wainger \cite{nagel}, Parmeggiani \cite{parmeggiani}, S\'anchez-Calle \cite{sanchez}). Roughly speaking, their (deep) result states that the order of magnitude of the bottom of the spectrum of such an operator is given by
\begin{equation}\label{eq feff}
\inf_{\mathcal{B}} \sup_{(x,\omega)\in \mathcal{B}} a(x,\omega),
\end{equation}
where the infimum is taken over a family of testing boxes $\mathcal{B}$ of measure $1$ in phase space, conveniently associated with $a$. Hence, the geometry of the phase space is tailored to the given symbol. Notice that in Theorem \ref{teomain} the metric is, instead, fixed in advance, which makes things simpler. On the other hand, since $w$ could be singular, in \eqref{eq iup} the average of $w$ appears (instead of the supremum in \eqref{eq feff}).

As strategy, we will prove a suitable lower bound for the $p$-subLaplacian with potential $w$, at sufficiently small scales, and then we will come back to the original problem by using the reproducing property of $V_\phi$. This (perhaps unexpected) connection with a problem for subLaplacians will enable us to use some ideas and tools from \cite{fefferman} and also from Jerison \cite{jerison}, where lower bounds were proved for partial differential equations with constant coefficients (see H\"ormander \cite[Theorem 3.2.5]{hormander_book} for an early instance of this philosophy, and also \cite{shen}). 

Notice that the use of weighed $L^p$ norms as a measure of the time-frequency concentration dates back at least to Cowling and Price \cite{cowling} and Daubechies \cite{daubechies}. 
Moreover, optimization problems similar to \eqref{eq min pro}, in the case of the short-time Fourier transform and the wavelet transform (when $\phi$ is a Gaussian function or a Cauchy wavelet) have recently attracted great interest, as demonstrated by an impressive amount of work in the last two years (see
\cite{abreu2021donoho,frank2023sharp,galbis2022,GGRT,knutsen,ortega,trapasso,nicolatilli_fk,nicolatilli_norm,ramos,riccardi2023} and also Seip \cite{seip} for early work in this connection). We emphasize that Theorem \ref{teomain} is new even for these particular transforms. 
Also, we point out that uncertainty inequalities of a different kind, on unimodular Lie groups, were also addressed by a number of authors (see  \cite{astengo,ciatti2015,ciatti2007,dallara_trevisan,martini2010,ruzhansky2017} and the references therein). 

\subsection{Applications to Anti-Wick operators}
As an application of Theorem \ref{teomain} (in the case of the Schr\"odinger representation of the reduced Heisenberg group) we find an explicit formula for the order of magnitude of the bottom of the spectrum and of the essential spectrum of anti-Wick operators in $\bR^d$ with nonnegative symbols. This  is the analog, for anti-Wick operators, of the above mentioned formula by Fefferman and Phong \cite{fefferman_phong} for Weyl pseudodifferential operators and represented also the original motivation for this note. 

Anti-Wick operators (also named localization, or Toeplitz, or FBI operators in the literature) were introduced by Berezin \cite{berezin}, motivated by the mathematical formalism of quantum field theory (second quantization), and correspond to a certain ordering of the creation and annihilation operators (see Remark \ref{rem shubin}). Roughly speaking, to a function $a(x,\omega)$ in phase space one associates an operator $A_{a,h}$ on $L^2(\bR^d)$ given by 
\[
A_{a,h} f=(2\pi{{h}})^{-d}\int_{\bR^{2d}}a(x,\omega)\langle f, \varphi^{{h}}_{(x,\omega)}\rangle \varphi^{{h}}_{(x,\omega)} \,dx\,d\omega,
\]
where the coherent states $\varphi^{{h}}_{(x,\omega)}$ are defined in \eqref{eq cs}
(see Section \ref{sec anti-wick} for the precise definition, also for more general coherent states). Anti-Wick operators appear naturally in semiclassical analysis, complex analysis, mathematical physics and mathematical signal processing (see e.g. \cite{berezin_shubin,daubechies,degosson_book,teofanov,zhu,zworski}), but they play also a fundamental role in the study of lower bounds of pseudodifferential operators in view of the fact that nonnegative symbols give rise to nonnegative operators; see the pioneering paper by C\'ordoba and Fefferman \cite{Cordoba} and Lerner's book \cite{lerner}. However, because of the uncertainty principle, one expects that the spectrum is in fact separated from $0$. For example, although the symbol $|x|^2+|\omega|^2$ vanishes at the origin, in \eqref{eq uncertainty} we have a lower bound of order $h$. Here we are interested in estimating, in general, {\it the order of magnitude of this spectral gap}. It turns out, perhaps surprisingly, that a complete answer can be given for several symbol classes without any assumption on the zero set of the symbol. We can summarize our results for Muckenhoupt symbols as follows (see Theorem \ref{mainteo2} and Corollary \ref{maincor}). 

We denote by $B((x_0,\omega_0),r)$ the Euclidean ball in $\bR^{2d}$ with center $(x_0,\omega_0)$ and radius $r$.\par\medskip
{\it Suppose that the symbol $a$ belongs to the (Euclidean) Muckenhoupt class $A_{\infty}(\bR^{2d})$ (see Definition \ref{def muck2}) and let $\sigma(A_{a,h})$ and $\sigma_{\rm ess}(A_{a,h})$ be the spectrum and essential spectrum of $A_{a,h}$, respectively. Let 
\[
\boldsymbol{\lambda}(a,h):=\inf_{(x_0,\omega_0)\in\bR^{2d}} \fint_{B((x_0,\omega_0),{h}^{1/2})} a(x,\omega)\, dx\,d\omega
\]
and 
\[
\boldsymbol{\lambda}_{\rm ess}(a,h):=\liminf_{(x_0,\omega_0)\to\infty}\fint_{B((x_0,\omega_0),{h}^{1/2})} a(x,\omega)\, dx\,d\omega. 
\]
Then we have 
\begin{equation}\label{eq 7mar1} 
C^{-1} \boldsymbol{\lambda}(a,h)\leq \inf \sigma(A_{a,h}) \leq C \boldsymbol{\lambda}(a,h)
\end{equation} 
and 
\begin{equation}\label{eq 7mar2}
C^{-1} \boldsymbol{\lambda}_{\rm ess}(a,h)\leq \inf \sigma_{\rm ess}(A_{a,h})\leq C \boldsymbol{\lambda}_{\rm ess}(a,h),
\end{equation}
for a suitable constant $C>0$ depending only on the $A_\infty$ bound of $a$.
}\par\medskip
Some remarks are in order.
\begin{remark}\ 
\begin{itemize}
\item[(a)] The class $A_\infty(\bR^{2d})$ contains every nonnegative (nonzero) polynomial, which represented a case of special interest in \cite{berezin} (see Remark \ref{rem shubin} and Corollary \ref{cor poli}). The motivating example is given by $a(x,\omega)=|x|^2+|\omega|^2$, for which $\boldsymbol{\lambda}(a,h)=C_dh$ (for a constant $C_d$ depending on the dimension). In this case \eqref{eq 7mar1} agrees with \eqref{eq uncertainty}. 
\item[(b)] The estimate \eqref{eq 7mar1} is interesting mainly when ${\rm ess\,inf}_{\bR^{2d}}\, a=0$, in which case the term $\boldsymbol{\lambda}(a,h)$ represents the quantum correction to the ``classical'' ground energy (that is $0$). In general one can of course apply the above estimate to $a-{\rm ess\,inf}_{\bR^{2d}}\, a$ (provided this symbol belongs to $A_\infty(\bR^{2d}$)); see Corollary \ref{cor 7mar}. 
\item[(c)] As a consequence of \eqref{eq 7mar2}, 
{\it $A_{a,h}$ has purely discrete spectrum (i.e. consisting of isolated eigenvalues of finite multiplicity) if and only if }
 \[
 \lim_{(x_0,\omega_0)\to\infty}\int_{B((x_0,\omega_0),{h}^{1/2})} a(x,\omega)\, dx\,d\omega=+\infty
 \]
(here one could equivalently consider balls of any fixed radius $R>0$). 
  This is a refinement of the well known sufficient condition ``$a(x,\omega)\to+\infty$ as $(x,\omega)\to\infty$" in Shubin \cite{shubin1972} (which is valid, however, for more general symbols). 
\end{itemize}
\end{remark}
In Section \ref{sec semiclassicalbis} we will show similar, optimal estimates for semiclassical symbol classes, modulo an error $O(h^N)$ with $N$ arbitrarily large ($h\in (0,1]$). Again, no assumption is made on the zero set of the symbol.


 \subsection{Relation to Schr\"odinger operators}
The above problems are related, at least in spirit, to the {\it positivity and discreteness of the spectrum} of Schr\"odinger operators. These issues have a long and distinguished tradition going back at least to the work by Mol\v{c}anov \cite{molchanov} (where the average of the potential over cubes going to infinity already appeared); see also the already cited work by Fefferman and Phong \cite{fefferman_phong,fefferman}, Maz'ya and Shubin \cite{mazya} and Simon \cite{simon}. In the case of Muckenhoupt potentials the problem of the discretness of the spectrum was addressed by Dall'Ara \cite{dallara}  (see also Auscher and Ben Ali \cite{auscher}), whereas the phase space localization of the eigenfunctions is currently object of study by Decio, De Dios Point, Malinnikova and Mayboroda \cite{decio}  (see also Arnold, David, Filoche, Jerison and Mayboroda \cite{arnold2019} and the references therein). Also, recently Bruno and Calzi \cite{bruno_calzi} considered a left-invariant subLaplacian on a noncompact connected Lie group with a potential $V\in A_{\infty,{\rm loc}}(G,{\rm d}_C,\lambda)$ (with the above notation, where $\lambda$ is a left Haar measure). A characterization of the discreteness of the spectrum was provided under the hypothesis that the absolutely continuous measure that has density $V$ with respect to $\lambda$ is locally doubling (cf. \eqref{eq loc doubl}). As a byproduct of the results in Section \ref{sec prel} we will see that, in fact, this condition is always verified for such a potential (Corollary \ref{pro loc doubl}). Hence their characterization holds for every $V\in A_{\infty,{\rm loc}}(G,{\rm d}_C,\lambda)$;  we will briefly comment on this in the Appendix. 
\subsection{Further developments} 
The above results suggest a number of far-reachnig generalizations. First, it would be natural to extend the above estimates to rougher symbols, in terms of a convenient notion of capacity, as in the work by Maz'ya and Shubin \cite{mazya} on the positivity and discreteness of the spectrum of Schr\"odinger operators. Secondly, it is likely that the same strategy as in this note can be used for reproducing kernel Hilbert spaces of holomorphic functions, such as weighted Bergman and Hardy spaces (indeed, the above results for anti-Wick operators can be easily rephrased in terms of Toeplitz operators on the Fock space \cite{zhu}). Also, similar estimates for higher eigenvalues (below the essential spectrum) are certainly worth studying, as well as the case of symbols not bounded below, in the spirit of Fefferman's results \cite{fefferman} for Schr\"odinger operators. We decided to postpone to a subsequent work a systematic investigation of these issues, which would take us too far.

\section{Preliminary results}\label{sec prel}
\subsection{Left-invariant subRiemannian structures} \label{sec lis}
We use the same notation as in the introduction. Hence, $G$ is a connected Lie group, $X_1,\ldots,X_\ell $ are linearly independent left-invariant vector fields satisfying H\"ormander's condition (that is, their iterated commutators span the Lie algebra as a real vector space) and ${\rm d}_C(\cdot,\cdot)$ is the corresponding Carnot–Carath\'eodory distance, defined as follows.

A piecewise $C^1$ curve $\gamma:[0,T]\to G$ is said {\it horizontal} if $\gamma'(t)=\sum_{j=1}^\ell c_j(t)X_j$ for a.e.\ $t$, for some functions $c_j\in L^\infty([0,T])$, $j=1,\ldots,\ell$. We define the length of such a curve as 
\[
L(\gamma)=\int_0^T \Big(\sum_{j=1}^\ell c_j(t)^2\Big)^{1/2}\,dt.
\]
For $x,y\in G$, by the H\"ormander's condition and the connectedness of $G$ there is at least one horizontal curve $\gamma$ joining $x$ and $y$ ($\gamma(0)=x$, $\gamma(T)=y$) and we define 
\[
{\rm d}_C(x,y)=\inf_\gamma L(\gamma)
\]
where the infimum is taken over all these horizontal curves. It turns out that ${\rm d}_C(\cdot,\cdot)$ is a left-invariant distance on $G$ -the Carnot-Carath\'eodory distance associated with the vector fields $X_1,\ldots,X_\ell$ --- and is continuous as a function on $G\times G$. Indeed, it induces on $G$ the same manifold topology and the balls are relatively compact. Moreover, for any $x,y\in G$ there exists a length minimizing horizontal curve $\gamma$ joining $x$ and $y$, namely with $L(\gamma)={\rm d}_C(x,y)$. Such a curve will be called a geodesic. We refer to \cite{agrachev_book} for a proof of all these facts.  

We denote by $\mu$ a right Haar measure of $G$, and we write $\mu_\chi$ for the absolutely continuous measure that has density $\chi$ with respect to $\mu$, where $\chi:G\to(0,+\infty)$ is any continuous character. We have 
\[
0<\mu_\chi(B(x,r))<\infty,\qquad x\in G,\ r>0,
\]
where $B(x,r):=\{y\in G: {\rm d}_C(x,y)<r\}$. 

While $\mu_\chi$ is not left-invariant, in general, it is relatively invariant in the sense that, for every Borel subset $E\subset G$ and $x\in G$,
\[
\mu_\chi(xE)=\chi(x)\Delta(x^{-1})\mu_\chi(E)
\]
where $\Delta(x)$ is the modular function.

It follows from \cite{nagel} that $\mu_\chi$ is locally doubling, in the sense that, for every $R>0$, 
\begin{equation}\label{eq loc doubl}
D(R):=\sup_B\frac{\mu_\chi(2B)}{\mu_\chi(B)}<\infty,
\end{equation}
where the supremum is taken over the balls of radius $r\in (0,R]$, and $2B$ denotes the ball with the same center as $B$ and double radius.

\subsection{Reproducing property} 
We have defined in \eqref{eq Vphi} the generalized wavelet transform $V_\phi$ associated with a square-integrable irreducible unitary representation $\pi:G\to\mathcal{U}(\mathcal{H})$, where $\mathcal{H}$ is a Hilbert space. The vector $\phi\in\mathcal{H}$, with $\|\phi\|=1$, satisfies the admissibility condition \eqref{eq admissibility}, where $\lambda$ is the left Haar measure.  As already observed there, $V_\phi:\mathcal{H}\to L^2(G,\lambda)$ is an isometry (\cite[Theorem 7.2]{wong}). Hence, writing \[
V_\phi f= V_\phi V_\phi^\ast V_\phi f,
\] and using that $$V_\phi^\ast F=\frac{1}{\sqrt{c_\phi}}\int_G F(x) \pi(x)\phi \, d\lambda$$ for $F\in L^2(G,\lambda)$ (where $c_\phi$ is defined in \eqref{eq admissibility}), we obtain the reproducing property 
\begin{equation}\label{eq due0}
V_\phi f(x)=\frac{1}{\sqrt{c_\phi}} V_\phi f\ast V_\phi\phi(x), 
\end{equation}
(cf. \cite[Theorem 7.6]{wong} and also \cite{feichtinger}), where the convolution of two functions $u,v$ on $G$ is defined as 
\[
u\ast v(x)=\int_G u(y)v(y^{-1}x)\,d\lambda(y)=\int_G u(xy)v(y^{-1})\,d \lambda(y).
\] 
Also, observe that if $X$ is a left-invariant vector field, 
\[
X(u\ast v)= u\ast Xv.
\]
We also observe that
\[
V_\phi\phi(x^{-1})=\overline{V_\phi\phi(x)}.
\]
Hence, if $X$ is a (real) left-invariant vector field on $G$, 
\begin{equation}\label{eq due}
X(V_\phi f\ast V_\phi\phi)(x)=\int_G V_\phi f(xy) \overline{XV_\phi \phi(y)} d \lambda(y). 
\end{equation}

\subsection{Poincar\'e inequality}
We will need the local Poincar\'e inequality in the following form. For any ball $B\subset G$ and $u\in L^1_{\rm loc}(G)$, we write 
\begin{equation}\label{eq ub}
u_B=\fint_B u \, d\mu_\chi:=\frac{1}{\mu_\chi(B)}\int_B u \, d\mu_\chi
\end{equation}
for the average of $u$ over $B$. Moreover, let $\nabla u=(X_1 u,\ldots,X_\ell u)$ be the {\it horizontal} gradient, and set $|\nabla u|=(\sum_{j=1}^\ell |X_j u|^2)^{1/2}$. 
\begin{theorem}[Poincar\'e inequality]\label{eq teo poincare}
    Let $1\leq p<\infty$. For every $R>0$ there exists a constant $C_R>0$ such that, for every ball $B$ of radius $r\in (0,R]$ and every $u\in L^p_{\rm loc}(G,\mu_\chi)$, with $\nabla u\in L^p_{\rm loc}(G,\mu_\chi)$, 
\begin{equation}\label{eq poincare}
    \int_B |\nabla u|^p \,d\mu_\chi\geq C_R r^{-p}\|u-u_B\|_{L^p(B,\mu_\chi)}^p.
    \end{equation}
    \end{theorem}
    \begin{proof} The inequality \eqref{eq poincare} was proved in \cite[Theorem 3.1]{bruno} for a general noncompact connected Lie group $G$, assuming $u\in C^\infty(G)$ (see also \cite{dallara_trevisan} for a new, beautiful approach to Poincar\'e inequalities on groups). The same proof in fact holds for $G$ compact (and connected).
    
     For $u$ as in the statement, when proving \eqref{eq poincare} on a given ball $B$, we can suppose that $u$ has compact support, hence $u,\nabla u\in L^p(G,\mu_\chi)$. Now, \eqref{eq poincare} holds for $u_\varepsilon=\varphi_\varepsilon \ast u$,
   if $\varphi_\varepsilon$ is a {\it smooth} approximate identity for $L^1(G,\lambda)$ (see e.g. \cite[Section 2.2]{bruno_calzi}). On the other hand, as $\varepsilon\to0$, 
   \[
   (u_\varepsilon)_B=u_B\int_G \varphi_\varepsilon(y)\chi(y)\Delta(y^{-1})\,d \lambda(y)\to u_B
   \]
    (where $\Delta$ is the modular function) because $\chi(e)=\Delta(e)=1$ (where $e$ is the unit of $G$). Moreover $u_\varepsilon\to u$ in $L^p(G,\mu_\chi)$ and $\nabla u_\varepsilon=\varphi_\varepsilon\ast \nabla u\to \nabla u$ in $L^p(G,\mu_\chi)$. Therefore we have the desired conclusion. 
\end{proof}
\subsection{Local doubling property}
We collect here some auxiliary results related to the local doubling property of $\mu_\chi$ (cf. \eqref{eq loc doubl}).
We begin with the following fact.
\begin{proposition}\label{pro cont} For every $x_0\in G$, the function $r\mapsto \mu_\chi(B(x_0,r))$ is continuous on $(0,\infty)$. 
\end{proposition}
\begin{proof}
By monotone convergence and since the balls have finite measure, the possible jump of the function in the statement, at $r=R$, equals $\mu_\chi(E_R)$, where $E_R:=\{y\in G:\ {\rm d}_C(x_0,y)=R\}$. Hence we have to prove that $\mu_\chi(E_R)=0$. 

To this end, we first observe that in the present setting the differentiation theorem holds true: if $f\in L^1_{\rm loc}(G,\mu_\chi)$,
\begin{equation}\label{eq diff}
\lim_{r\to0}\fint_{B(x,r)}f\,d\mu_\chi=f(x) \qquad a.e.\ x\in G.
\end{equation}
This can be seen by a standard argument. Namely, it is clear that this holds if $f$ is continuous. The extension to $f\in L^1_{\rm loc}(G,\mu_\chi)$ follows by the density of the continuous functions with compact support in $L^1(G,\mu_\chi)$ ($\mu_\chi$ is a Radon measure), and the weak type inequality 
\[
\mu_\chi(\{x\in G:\ M_1 f(x)>t\})\leq \frac{C}{t}\|f\|_{L^1(G,\mu_\chi)} \qquad t>0
\]
 for the {\it restricted} centered maximal operator 
\[
M_1 f(x):=\sup_{0<r\leq 1} \fint_{B(x,r)}|f|\,d\mu_\chi.
\]
 The latter inequality is in turn a consequence of the local doubling property of $\mu_\chi$ and the Vitali covering argument; see e.g. \cite[Theorem 3]{stromberg}.  
 
 Now suppose, by contradiction, that $\mu_\chi(E_R)>0$. By applying \eqref{eq diff} to $f=\chi_{E_R}$  we see that there exists $y_0\in E_R$ with the following property: for every $\delta\in(0,1)$ there is $0<r_\delta<\min\{R,1\}$ such that
 \begin{equation}\label{eq rdelta1}
 \mu_\chi(E_R\cap B(y_0,r_\delta))\geq \delta \mu_\chi(B(y_0,r_\delta)).
 \end{equation}
 Let $\gamma$ be a geodesic joining $x_0$ and $y_0$, hence $L(\gamma)=R$. By the continuity of the Carnot–Carath\'eodory distance there exists a point $z_0$ on this curve such that ${\rm d}_C(z_0,y_0)=r_\delta/2$. Clearly, ${\rm d}_C(x_0,z_0)=R-r_\delta/2$, because $\gamma$ is a geodesic. Now, we have $B(z_0,r_\delta/4)\subset B(y_0,r_\delta)$ and
 \begin{equation} \label{eq rdelta2}
 \mu_\chi(B(z_0,r_\delta/4))\geq C_0 \mu_\chi(B(z_0,2r_\delta)) \geq C_0\mu_\chi(B(y_0,r_\delta))
 \end{equation}
 where $C_0=D(1)^{-3}$, cf. \eqref{eq loc doubl}. 
 
 Choosing $\delta>1-C_0$ we see from \eqref{eq rdelta1} and \eqref{eq rdelta2} that $E_R\cap B(z_0,r_\delta/4)\not=\emptyset$. Let $v_0\in E_R\cap B(z_0,r_\delta/4)$. Then 
 \[
 R={\rm d}_C(x_0,v_0)\leq {\rm d}_C(x_0,z_0)+{\rm d}_C(z_0,v_0)< R-\frac{r_\delta}{2}+\frac{r_\delta}{4}<R,
 \]
 which is a contradiction. 
\end{proof}
\begin{proposition}\label{pro cov} For every $R>0$ there exists a constant $N=N_R>0$ such that, for every $r\in (0,R]$, there exists a covering of $G$ made of balls $B_j$ of radius $r$ and such that $\sum_j \chi_{B_j}\leq N$. 
\end{proposition}
\begin{proof}
Consider a maximal family of pairwise disjoint balls $B(x_j,r/2)$. Then the balls $B_j:=B(x_j,r)$ cover $G$. If $\overline{x}$ belongs to, say, $B(x_j,r)$, with $j=1,\ldots, N$, then $B(x_j,r/2)\subset B(\overline{x},2r)$. On the other hand, since the left Haar measure $\lambda$ is locally doubling ($\lambda=\Delta\mu$, where $\Delta$ is the modular function) we have 
\[
\lambda(B(x_j,r/2))\geq c_R  \lambda(B(x_j,2r))= c_R  \lambda(B(\overline{x},2r)),
\]
and therefore $Nc_R\leq 1$. 
\end{proof}
\subsection{Muckenhoupt property} We prove some properties of weights in the local Muckenhoupt class $A_{\infty,{\rm loc}}^{\alpha}(G,{\rm d}_C,\mu_\chi)$, as defined in Definition \ref{def muck}. This complements the analysis in \cite{bruno_calzi}. 

The following  basic fact will be crucial in the proof of our main result.
\begin{lemma}\label{lemma lem1}
    Let $\alpha\in(0,1)$. For every $w\in A_{\infty,{\rm loc}}^{\alpha}(G,{\rm d}_C,\mu_\chi)$ we have  
    \[
\mu_\chi\Big(\Big\{x\in B:\ w(x)\geq \alpha \fint_B w\,d\mu_\chi \Big\}\Big)\geq \alpha\mu_\chi(B)
    \]
    for every ball $B$ of radius $r\in (0,1]$. 
\end{lemma}
\begin{proof}
If the set
\[
E:=\Big\{x\in B:\ w(x)< \alpha \fint_B w\,d\mu_\chi \Big\}
\]
had measure $\mu_\chi(E)> (1-\alpha) \mu_\chi(B)$, by \eqref{eq muck} we would have   
\[
\alpha\geq \frac{\alpha\mu_\chi(E)}{\mu_\chi(B)}> \frac{\int_{E}w\, d\mu_\chi}{\int_B w\, d\mu_\chi}\geq\alpha,
\]
which is a contradiction.
\end{proof}
The following result completes the analysis in \cite{bruno_calzi} about the relationships between several definitions of the local Muckenhoupt class $A_{\infty,{\rm loc}}(G,{\rm d}_C,\mu_\chi)$ and provides, indeed, a characterization analogous to that valid in the Euclidean case (cf. \cite[Theorem 9.3.3]{grafakos}). 

For $R>0$, let $\mathcal{B}_R$ be the collection of open balls of radius $R$ in $(G,{\rm d}_C)$. For $p\in (1,\infty)$, we denote by $p'=p/(p-1)$ its conjugate exponent. 

\begin{proposition}  \label{pro equivalenze} Let $w\in L^1_{\rm loc}(G,\mu_\chi)$, with $w>0$ almost everywhere. Let $R>0$. The following conditions are equivalent.
\begin{itemize}
\item[(1)]There are $\varepsilon,\delta\in(0,1)$ such that for every ball $B\in\mathcal{B}_R$ and every Borel subset $F\subset B$,
\[
\mu_\chi(F)\leq \varepsilon\mu_\chi(B) \Longrightarrow \int_F w\,d\mu_\chi \leq \delta \int_B w\,d\mu_\chi.
\]
\item[(2)] There are $p\in(1,\infty)$ and $C>0$ such that, for every ball $B\in\mathcal{B}_R$ and every Borel subset $F\subset B$,
\[
\frac{ \int_F w\,d\mu_\chi}{\int_B w\,d\mu_\chi}\geq C\Big(\frac{\mu_\chi(F)}{\mu_\chi(B)}\Big)^p.
\]
\item[(3)] There are $p\in(1,\infty)$ and $C>0$ such that, for every ball $B\in\mathcal{B}_R$,
\[
\Big(\fint_B w\,d\mu_\chi\Big)\Big(\fint_B w^{-p'/p}d\mu_\chi\Big)^{p/p'}\leq C.
\]
\item[(4)] There are $\delta,c>0$ such that, for every ball $B\in\mathcal{B}_R$,
\[
\mu_\chi\Big(\Big\{x\in G:\ w(x)\geq \delta \fint_B w\,d\mu_\chi\Big\}\Big)\geq c\mu_\chi(B).
\]
\item[(5)] $w\in A_{\infty,{\rm loc}}(G,{\rm d}_C,\mu_\chi)$.
\end{itemize}  
More precisely, in the implication $(1)\Rightarrow (2)$ one can choose the constants $p,C$ in $(2)$ depending only on $R$ and the constants $\varepsilon,\delta$ in $(1)$, and similarly for all the other implications. 

In particular, if one of the conditions $(1),\ (2),\ (3),\ (4)$ holds for some $R$ then all these conditions hold for every $R$. 
\end{proposition}

\begin{proof}
The desired equivalences were proved in \cite{bruno_calzi}, except for the implication $(4)\Rightarrow (3)$, that is new and whose proof makes use of Proposition \ref{pro cont} above. For the sake of clarity we give precise references. 

It was proved in \cite[Proposition 6.1]{bruno_calzi} that $(3)\Rightarrow (2)\Rightarrow (1)$. On the other hand, the proof of Proposition \ref{lemma lem1} (taking $B\in\mathcal{B}_R$) shows that $(1)\Rightarrow (4)$. Also, in \cite[Propositions 6.3 and 6.5]{bruno_calzi} it was proved that if $(4)$ holds for some $R>0$ then $(3)$ holds \textit{for all $R$}, provided that the function $r\mapsto \mu_\chi(B(x_0,r))$ is continuous on $(0,R']$ for some $R'>0$. On the other hand we know that this latter condition is always verified, by Proposition \ref{pro cont}. Hence, if one of the conditions (1), (2), (3), (4) holds for some $R$ then all those conditions hold for every $R>0$. 

Finally, $(5)$ is clearly equivalent to $(1)$ for $R=1$.

The remark about the dependence of the constants follows by an inspection of the proofs of the aforementioned results of \cite{bruno_calzi}.

We observe that the results of \cite{bruno_calzi} were actually proved when $\chi$ is the modular function of $G$, hence $\mu_\chi$ is a left Haar measure. However, an inspection of those proofs reveals that they extend with obvious changes to every measure of the type $\mu_\chi$. 
\end{proof}
\begin{corollary}\label{pro loc doubl}
Let $\alpha\in (0,1)$. For every $R>0$ there exists a constant $C=C(\alpha,R)$ such that, for every $w\in A^{\alpha}_{\infty,{\rm loc}}(G,{\rm d}_C,\mu_\chi)$ we have 
\[
\int_{2B} w\, d\mu_\chi\leq C \int_{B} w\, d\mu_\chi
\]
for all balls $B$ of radius $r\in (0,R]$.
 
\end{corollary}
\begin{proof}
This result follows from Proposition \ref{pro equivalenze}. Alternatively, it follows by combining Proposition \ref{pro cont} with \cite[Propositions 6.5 and 6.3]{bruno_calzi}. 
\end{proof}
For $r\in (0,1]$, let 
\begin{equation}\label{eq cr}
C_r(w)=\inf\Big\{\fint_{B} w\,d\mu_\chi:\ B\textrm{ is a ball of radius }r\Big\}.
\end{equation}
Observe that $C_1(w)=I_{\rm UP}(w)$ as defined in \eqref{eq iup}. 
\begin{proposition}\label{pro cr} Let $\alpha\in(0,1)$. There exist positive constants $M_1=M_1(\alpha)$, $\kappa_1=\kappa_1(\alpha)$ and $M_2,\kappa_2$ (depending only on the constant $D(1/2)$ in \eqref{eq loc doubl}) such that, for $0<r'\leq r\leq 1$ and $w\in A_{\infty,{\rm loc}}^{\alpha}(G,{\rm d}_C,\mu_\chi)$,
\[
M_1\Big(\frac{r}{r'}\Big)^{-\kappa_1}C_r(w)\leq C_{r'}(w)\leq M_2\Big(\frac{r}{r'}\Big)^{\kappa_2}C_r(w)
\]
\end{proposition}
\begin{proof}
Let us prove the first inequality. Let $B'$ and $B$ be open balls of radius $r'$ and $r$, respectively, with the same center. We use Corollary \ref{pro loc doubl} with $R=1$, and denote by $C=C(\alpha,1)$ the constant that appears there. Let $n\geq0$ be the integer such that $2^n r'\leq r< 2^{n+1}r'$. 

By Corollary \ref{pro loc doubl},  for $w\in A_{\infty,{\rm loc}}^{\alpha}(G,{\rm d}_C,\mu_\chi)$,
\[
\int_B w\,d\mu_\chi \leq \int_{2^{n+1}B'} w\,d\mu_\chi\leq C^{n+1}\int_{B'}w\, d\mu_\chi.
\]
Let $\kappa_1$ be the smallest integer such that $C\leq 2^{\kappa_1}$ (since $C\geq 1$ we have $\kappa_1\geq0$). Then $C^{n+1}\leq (2^{n+1})^{\kappa_1}\leq 2^{\kappa_1} (r/r')^{\kappa_1}$ and therefore 
\[
\int_B w\,d\mu_\chi \leq 2^{\kappa_1} \Big(\frac{r}{r'}\Big)^{\kappa_1}\int_{B'}w\, d\mu_\chi.
\]
By dividing by $\mu_\chi(B)$ and using $\mu_\chi(B)\geq \mu_\chi(B')$ we obtain the first inequality. 

Concerning the second inequality we observe that, by \eqref{eq loc doubl} with $R=1/2$, denoting by $D=D(1/2)$ the constant that appears there, we have (with the same notation as above) 
\[
\frac{1}{\mu_\chi(B')}\leq \frac{1}{\mu_\chi(2^{-n-1} B)}\leq \frac{D^{n+1}}{\mu_\chi(B)}.
\]
Let $\kappa_2$ be the smallest integer such that $D\leq 2^{\kappa_2}$ (since $D\geq 1$ we have $\kappa_2\geq0$). Arguing as above we arrive at 
\[
\frac{1}{\mu_\chi(B')}\leq 2^{\kappa_2} \Big(\frac{r}{r'}\Big)^{\kappa_2} \frac{1}{\mu_\chi(B)}.
\]
 Multiplying by $\int_{B'} w\, d\mu_\chi$ and using $ \int_{B'} w\, d\mu_\chi\leq \int_{B} w\, d\mu_\chi$ gives the second inequality. 
\end{proof}

\subsection{The Schr\"odinger representation of the reduced Heisenberg group} \label{sec heisenberg} 
The reduced Heisenberg group can be realized as the product $\mathbb{H}^d:=\bR^d\times \bR^d \times \bR/2\pi \mathbb{Z}$ endowed with the product 
\[
(x,\omega,t)\cdot(x',\omega',t')=(x+x',\omega+\omega',t+t'-x\omega').
\]
It turns out to be a connected unimodular Lie group, and a Haar measure is just the Lebesgue measure (identifying $\mathbb{H}^d$ with $\bR^d \times \bR^d\times [0,2\pi)$). A basis of the vector space of left-invariant vector fields is given by
\[
X_j=\frac{\partial}{\partial x_j},\qquad \Omega_j=\frac{\partial}{\partial \omega_j}-x_j\frac{\partial}{\partial t},\qquad T=\frac{\partial}{\partial t}
\]
 for $j=1,\ldots,d$.
 
 The representation $\pi$ of $\mathbb{H}^{d}$ on $L^2(\bR^d)$ given by 
 \begin{equation}\label{eq pih}
 \pi(x,\omega,t) f(y)=e^{i t} e^{i \omega y} f(y-x),
 \end{equation}
 is the {\it Schr\"odinger representation} of $\mathbb{H}^{d}$. It is unitary, irreducible and square-integrable. Indeed, every $\varphi\in L^2(\bR^d)$, with $\|\varphi\|_{L^2}=1$, is admissible (cf. \eqref{eq admissibility}), since
 \[
 (2\pi)^{-(d+1)}\int_{\mathbb{H}^{d}} |\langle \varphi,\pi(x,\omega,t)\varphi \rangle |^2 dx\, d\omega\, dt=\|\varphi\|^4_{L^2}. 
 \]
 We refer to \cite[Chapter 17]{wong} for details. We observe that the function $\bR^{2d}\ni(x,\omega)\mapsto \langle f, \pi(x,\omega,0)\varphi\rangle$ is the so-called short-time Fourier transform of $f\in L^2(\bR^d)$ with window $\varphi$; see \cite{grochenig-book}. 
 
 In the following we will consider the class $A_{\infty,{\rm loc}}(\mathbb{H}^d)$ defined in Definition \ref{def muck}, with respect to the left-invariant Riemannian distance induced by the above vector fields, and the Haar measure. 
 
 \subsection{Euclidean Muckenhoupt class $A_\infty$} We recall the following definition of the $A_\infty$ class in $\bR^d$; cf. \cite[Chapter 9]{grafakos}. We denote by $|E|$ the Lebesgue measure of a set $E$. 
 
 \begin{definition}\label{def muck2}
 Let $\alpha\in(0,1)$. We denote by $A_{\infty}^{\alpha}(\bR^{d})$ the set of functions $w\in L^1_{\rm loc}(\bR^d)$, with $w>0$ almost everywhere, such that for all balls $B\subset\bR^d$ and all Lebesgue-measurable subsets $E\subset B$ we have 
\[
    |E|\geq (1-\alpha) |B| \Longrightarrow \int_E w(x)\, dx \geq \alpha \int_B w(x)\, dx.
\]
We also set $A_{\infty}(\bR^{d})=\cup_{0<\alpha<1}A_{\infty}^{\alpha}(\bR^{d})$. 
 \end{definition} 
 We will also use the notation $A_{\infty,{\rm loc}}^{\alpha}(\bR^{d})$ according to Definition \ref{def muck}, where $\bR^d$ is regarded as a metric measure space with the Euclidean distance and the Lebesgue measure. The following result will be useful later. 
 \begin{lemma}\label{lem equiv}
 For every $\alpha\in(0,1)$ there exists $\beta\in (0,1)$ such that, every $w\in A_{\infty,{\rm loc}}^\alpha(\bR^{2d})$, regarded as a function on $\mathbb{H}^d$ constant on the fibers, belongs to $A_{\infty,{\rm loc}}^\beta(\mathbb{H}^d)$.
 
 Moreover, for every  $r>0$ there exists $C=C(\alpha,r)>0$ such that for every $w$ as above, and $(x_0,\omega_0,t_0)\in\mathbb{H}^d$,
  \begin{align}\label{eq ineq2}
 C^{-1}\fint_{B((x_0,\omega_0),r)} w(x,\omega)\,dx\,d\omega&\leq \fint_{B((x_0,\omega_0,t_0),r)} w(x,\omega)\,dx\,d\omega\,dt\\
 &\leq C \fint_{B((x_0,\omega_0),r)} w(x,\omega)\,dx\,d\omega.\nonumber
 \end{align}
 
 \end{lemma}
 \begin{proof}
 The first statement follows by the following observations. By Proposition \ref{pro equivalenze} we can limit ourselves to check that the condition (3) in Proposition \ref{pro equivalenze} holds for balls of sufficiently small radius. It is also clear, since the class $A_{\infty,{\rm loc}}^\alpha(\bR^{2d})$ is translation invariant and $w$ is constant on the fibers, that by a left translation on $\mathbb{H}^d$ we can reduce ourselves to the case of balls centered at $e=(0,0,0)$ (the unit of $\mathbb{H}^d$). 
 
  Now, setting $B_r=B(e,r)\subset \mathbb{H}^d$ and $B'_r=B((0,0),r)\subset\bR^{2d}$ we easily see that there exist $r_0>0$, $c,C>0$ such that, for $r\leq r_0$,
 \[
 B'_{cr}\times (-cr,cr)\subset B_r\subset B'_{Cr}\times (-Cr,Cr).
 \]
(here $\mathbb{H}^d$ is identified with $\mathbb{R}^{2d}\times [-\pi,\pi)$, and we suppose $Cr_0<\pi$). Hence it is enough to check the condition (3) in Proposition \ref{pro equivalenze} with the balls $B=B_r$ replaced by $B'_{r}\times (-r,r)$,  with $r$ small enough, and the conclusion is then clear. 

The proof of the inequality \eqref{eq ineq2} is similar. Indeed, again by translation invariance of  the class $A_{\infty,{\rm loc}}^\alpha(\bR^{2d})$ and the fact the $w$ is constant along the fibers, it suffices to consider the balls $B_r$ and $B'_r$. Now, the above inclusion relationships (for $r=r_0$) give
\[
 \begin{aligned}
 2cr_0\int_{B'_{c{r_0}}} w(x,\omega)\,dx\,d\omega&\leq \int_{B_{r_0}} w(x,\omega)\,dx\,d\omega\,dt\\
 &\leq 2Cr_0 \int_{B'_{C{r_0}}} w(x,\omega)\,dx\,d\omega.\nonumber
 \end{aligned}
\]
The desired conclusion then follows by applying repeatedly Corollary \ref{pro loc doubl} (the measure of the balls that enters in the averages of $w$ can be absorbed in the multiplicative constant, that in 
\eqref{eq ineq2} is allowed to depend on $r$).
 \end{proof}
 \begin{remark}\label{rem covar}
 For future reference we notice the following elementary remarks, that are a consequence of the fact that the topology induced by the above Riemannian metric on $\mathbb{H}^d=\bR^{2d}\times\bR/2\pi\mathbb{Z}$ is the product topology and the projection $\mathbb{H}^d\to\bR^{2d}$ is a continuous and open group homomorphism. 
 \begin{itemize}
 \item[(a)] There exists $r_0\in(0,1]$ such that the projection on $\bR^{2d}$ of every ball $B((x_0,\omega_0,t_0),r_0)\subset\mathbb{H}^d$ is contained in the Euclidean ball $B((x_0,\omega_0),1)$.
 \item[(b)] There exists $r_1\in(0,1]$ such that the projection on $\bR^{2d}$ of every ball $B((x_0,\omega_0,t_0),1)\subset\mathbb{H}^d$ contains the Euclidean ball $B((x_0,\omega_0),r_1)$.
 \end{itemize}
 
 \end{remark}

\section{Proof of Theorem \ref{teomain}}
The following result can be regarded as a generalization of the ``Main Lemma" in \cite[page 146]{fefferman}, and reduces to it when $p=2$, $G=\bR^{n}$ with the Euclidean metric and Lebesgue measure, and $w$ is a (nonnegative) polynomial. \begin{lemma}\label{mainlemma}
Let $1\leq p<\infty$, $\alpha\in(0,1)$. There exists $C=C(p,\alpha)>0$ such that, for every ball $B$ of radius $r\in (0,1]$ and weight $w\in A_{\infty,{\rm loc}}^{\alpha}(G,{\rm d}_C,\mu_\chi)$ with $w_B\geq r^{-p}$, and every $u\in L^p_{\rm loc}(G,\mu_\chi)$, with $\nabla u\in L^p_{\rm loc}(G,\mu_\chi)$ we have
\begin{equation}
    \int_B (|\nabla u|^p+w |u|^p) \,d\mu_\chi\geq Cr^{-p}\|u\|_{L^p(B,\mu_\chi)}^p.  
\end{equation}   

\end{lemma}
\begin{proof}
By the Poincar\'e inequality (Theorem \ref{eq teo poincare}) we have
\[
\int_B |\nabla u|^p \,d\mu_\chi\geq 
   C r^{-p}\|u-u_B\|^p_{L^p(B,\mu_\chi)}.
\]
Hence it is sufficient to prove that 
\begin{equation}\label{eq tre}
\int_B (|\nabla u|^p+w |u|^p) \,d\mu_\chi\geq Cr^{-p}|u_B|^p\mu_\chi(B).
\end{equation}
To this end, again by the Poincar\'e inequality we have 
    \begin{equation}\label{eq poincare2}
   \int_B (|\nabla u|^p+w |u|^p) \,d\mu_\chi\geq 
   C r^{-p}\|u-u_B\|_{L^p(B,\mu_\chi)}^p+\int_B w |u|^p \,d\mu_\chi.
    \end{equation}
Hence, setting $E:=\{x\in B:\ w(x)\geq \alpha w_B \}$, by Lemma \ref{lemma lem1} we deduce that 
\begin{align*}
\int_B (|\nabla u|^p+w |u|^p) \,d\mu_\chi &\geq \int_E (C r^{-p}|u-u_B|^p+\alpha w_B |u|^p)\, d\mu_\chi\\
&\geq \int_E (C r^{-p}|u-u_B|^p+\alpha r^{-p} |u|^p)\, d\mu_\chi \\
&\geq \min\{C,\alpha\}r^{-p} \int_E |u-u_B|^p+ |u|^p)\, d\mu_\chi\\
&\geq 2^{-p+1}\min\{C,\alpha\}r^{-p}  |u_B|^p\mu_\chi(E)\\
&\geq 2^{-p+1}\alpha \min\{C,\alpha\} r^{-p} |u_B|^p\mu_\chi(B),
\end{align*}
which is \eqref{eq tre}.
\end{proof}
We can now prove Theorem \ref{teomain}.

\begin{proof}[Proof of Theorem \ref{teomain}]
We can suppose $I_{\rm UP}(w)>0$. Let $C_r(w)$ be the function defined in \eqref{eq cr}. Then by Proposition \ref{pro cr} we have $C_r(w)>0$ for every $r\in (0,1]$.  

We apply Lemma \ref{mainlemma} to the weight $r^{-p}w/C_r(w)$, which belongs to the class $A_{\infty,{\rm loc}}^{\alpha}(G,{\rm d}_C,\mu_\chi)$ for every $r\in(0,1]$.  
Hence, for every ball $B$ of radius $r\in(0,1]$ and $u\in L^p_{\rm loc}(G,\mu_\chi)$, with $\nabla u\in L^p_{\rm loc}(G,\mu_\chi)$ we have
\[ 
\int_B \Big(|\nabla u|^p+\frac{r^{-p} w}{C_r(w)} |u|^p\Big) \,d\mu_\chi\geq Cr^{-p}\|u\|_{L^p(B,\mu_\chi)}^p.
\]
By Proposition \ref{pro cov} there exists a covering of $G$ made of balls of radius $r$ and with a number of overlaps uniformly bounded with respect to $r$. Applying the above estimate to each of these balls and summing up we obtain
\begin{equation}\label{eq lowboun}
\int_G \Big(|\nabla u|^p+\frac{r^{-p} w}{C_r(w)} |u|^p\Big) \,d\mu_\chi \geq Cr^{-p}\|u\|_{L^p(G,\mu_\chi)}^p
\end{equation}
for a new constant $C>0$. 

Now, we apply this estimate with $
u=V_\phi f$
 for any $f\in\mathcal{H}$ with 
\[
\|V_\phi f\|_{L^p(G,d\mu_\chi)}=1. 
\]
Let us verify that $\nabla u\in L^p(G,d\mu_\chi)$. 

Using \eqref{eq due0}, \eqref{eq due} and Minkowski's inequality, together with 
\[
\int_G |u(xy)|^p\, d\mu_\chi(x)=\chi(y^{-1})\|u\|^p_{L^p(G,\mu_\chi)},
\]
 we see that, for $j=1,\ldots,\ell$,
\[
\|X_ju\|_{L^p(G,\mu_\chi)}= \frac{1}{\sqrt{c_\phi}}\|X_j(V_\phi f\ast V_\phi \phi) \|_{L^p(G,\mu_\chi)}\leq C'_j \|V_\phi f\|_{L^p(G,\mu_\chi)}=C'_j,
\]
where 
\[
C'_j:= \frac{1}{\sqrt{c_\phi}}\int_G \chi^{-1/p} |X_j V_\phi \phi| d\lambda<\infty
\]
by the assumption \eqref{eq phi}. 
As a consequence, 
\[
\int_G |\nabla u|^p\,d\mu_\chi\leq C'',
\]
which combined with \eqref{eq lowboun} yields 
\[
\int_G \frac{r^{-p} w}{C_r(w)} |u|^p \,d\mu_\chi\geq {C''}
\]
if $r$ is chosen small enough, so that $Cr^{-p}\geq 2C''$ (and $r\leq1$). Fixing $r$ in this way, we have 
\[
\int_G w |u|^p \,d\mu_\chi\geq {C''}C_r(w)r^p.
\]
By Proposition \ref{pro cr} we know that \[
C_r(w)\geq C'''C_1(w) = C''' I_{\rm UP}(w),\] for some constant $C'''>0$. This concludes the proof of Theorem \ref{teomain}. 
\end{proof}
\section{Anti-Wick operators}\label{sec anti-wick}
In this section we recall the definition and basic properties of anti-Wick operators. 
\subsection{Anti-Wick operators with a distribution symbol}
Let ${{h}}>0$ and $\varphi\in\cS(\bR^d)$, with $\|\varphi\|_{L^2(\bR^d)}=1$. Consider the {\it generalized coherent states}
\[
\varphi^{{{h}}}_{(x_0,\omega_0)}(x):={{h}}^{-d/4} e^{\frac{i}{{{h}}}\omega_0 x} \varphi({{h}}^{-1/2}(x-x_0)) \qquad x,x_0,\omega_0\in\bR^d.
\]
Observe that $\|\varphi^{{{h}}}_{(x_0,\omega_0)}\|_{L^2(\bR^d)}=1$ for every ${{h}}>0$, $x_0,\omega_0\in\bR^d$.  When $\varphi(x)=\pi^{-d/4}e^{\frac{1}{2}|x|^2}$ we recapture the particular case considered in Introduction; cf. \eqref{eq cs}. 

It is well known that the Parseval formula 
\begin{equation}\label{eq parseval} 
(2\pi{{h}})^{-d} \int_{\bR^{2d}}| \langle f,\varphi^{{{h}}}_{(x,\omega)}\rangle |^2\, dx\,d\omega= \|f\|^2_{L^2}
\end{equation} 
holds true for every $f\in L^2(\bR^{d})$. It is equivalent to the {\it resolution of the identity} formula:
\[
f=(2\pi{{h}})^{-d}\int_{\bR^{2d}}  \langle f,\varphi^{{{h}}}_{(x,\omega)}\rangle  \varphi^{{{h}}}_{(x,\omega)}\, dx\, d\omega,
\]
to be understood weakly in $L^2(\bR^d)$; see \cite[Proposition 235]{degosson_book} and \cite[Theorem 3.2.1]{grochenig-book}. This motives the following definition. 

\begin{definition}\label{def antiwick} Let $a\in\cS'(\bR^{2d})$. The anti-Wick operator ${\rm Op}^{\rm AW}_h(a):\mathcal{S}(\bR^d)\to \mathcal{S}'(\bR^d)$ with ${{h}}$-symbol $a$ is given by 
\begin{equation}\label{eq anti-w}
{\rm Op}^{\rm AW}_h(a)f=(2\pi{{h}})^{-d}\int_{\bR^{2d}} a(x,\omega) \langle f,\varphi^{{{h}}}_{(x,\omega)}\rangle \varphi^{{{h}}}_{(x,\omega)} \, dx\,d\omega,
\end{equation}
to be understood weakly in the sense of distributions, namely (with abuse of notation) for $f,g\in\mathcal{S}(\bR^d)$,
\[
\langle {\rm Op}^{\rm AW}_h(a) f,g\rangle = (2\pi{{h}})^{-d}\int_{\bR^{2d}} a(x,\omega) \langle f,\varphi^{{{h}}}_{(x,\omega)}\rangle\overline{\langle g,\varphi^{{{h}}}_{(x,\omega)}}\rangle\, dx\, d\omega.
\]
\end{definition}
It is easy to see that if $f\in\mathcal{S}(\bR^{d})$ the function $\langle f,\varphi^{{{h}}}_{(x,\omega)}\rangle$ is Schwartz, so that the above definition makes sense.  
\begin{remark} \label{rem shubin} 
To give the flavor of this quantization procedure, following \cite{berezin} we compute the anti-Wick operator associated with a polynomial symbol
\[
a(z):=\sum_{|\alpha|+|\beta|\leq m} c_{\alpha,\beta} z^\alpha\overline{z}^\beta,\quad z:=x+i\omega.
\]
Consider the annihilation and creation operators ${\bf a}_j$ and ${\bf a}^+_j$, $j=1,\ldots,d$, respectively, given (in $\mathcal{S}(\bR^{d})$ or even $\mathcal{S}'(\bR^{d}))$ by 
\[
{\bf a}_j={{h}}\frac{\partial}{\partial x_j} +x_j,\qquad {\bf a}^+_j=-{{h}}\frac{\partial}{\partial x_j} +x_j.
\]
Let us check that 
\[
{\rm Op}^{\rm AW}_h(a)=\sum_{|\alpha|+|\beta|\leq m} c_{\alpha,\beta}\, {\bf a}^\alpha({\bf a}^+)^\beta.
\]
We observe that, for $f\in \mathcal{S}(\bR^{d})$,
\[
(x_j+i\omega_j)\langle \varphi^{{h}}_{(x,\omega)},f \rangle = \langle \varphi^{{h}}_{(x,\omega)},{\bf a}_j^+ f \rangle
\]
and taking the conjugate of both sides, 
\[
(x_j-i\omega_j)\langle f,\varphi^{{h}}_{(x,\omega)} \rangle = \langle {\bf a}_j^+ f, \varphi^{{h}}_{(x,\omega)} \rangle.
\]
Hence, for $f,g\in \mathcal{S}(\bR^{d})$,
\begin{align*}
\langle {\rm Op}^{\rm AW}_h(a) f,g\rangle&=(2\pi{{h}})^{-d}\sum_{|\alpha|+|\beta|\leq m}c_{\alpha,\beta}\int_{\bR^{2d}} z^\alpha\overline{z}^\beta \langle f,\varphi^{{h}}_{(x,\omega)}\rangle \langle \varphi^{{h}}_{(x,\omega)},g\rangle\, dx\,d\omega\\
&=(2\pi{{h}})^{-d}\sum_{|\alpha|+|\beta|\leq m}c_{\alpha,\beta}\int_{\bR^{2d}} \langle {({\bf a}^+)}^\beta  f,\varphi^{{h}}_{(x,\omega)}\rangle \langle \varphi^{{h}}_{(x,\omega)},{({\bf a}^+)}^\alpha g\rangle\, dx\,d\omega\\
&=\sum_{|\alpha|+|\beta|\leq m}c_{\alpha,\beta}\langle {({\bf a}^+)}^\beta  f,{({\bf a}^+)}^\alpha g\rangle,
\end{align*}
which gives the desired conclusion. 

In particular, to the symbol $|x|^2+|\omega|^2=|z|^2$ appearing in \eqref{eq uncertainty} there corresponds the operator 
\[
\sum_{j=1}^d {\bf a}_j{\bf a}_j^+=-{{h}}^2\Delta+|x|^2 +d{{h}}. 
\]
Since the lowest eigenvalue of the harmonic oscillator $-{{h}}^2\Delta+|x|^2$ is $d{{h}}$, this agrees with the lower bound \eqref{eq uncertainty}. 

We also observe that the class of anti-Wick operators with polynomial symbols coincides with that  of differential operators with polynomial coefficients. 
\end{remark}
\subsection{Self-adjoint realizations via quadratic forms}\label{sec quadratic form}
Suppose now that \par\medskip
{\it $a(x,\omega)$ is a nonnegative measurable function in $\bR^{2d}$ with at most polynomial growth, in the sense that for some $C,N>0$,
 \begin{equation}\label{eq polgr}
\int_{B(0,r)} a(x,\omega)\, dx d\omega\leq C(1+r)^N\qquad\forall r>0.
\end{equation}
}
Consider the set 
\[
\mathcal{D}_{a,h}:=\{f\in L^2(\bR^d): \int_{\bR^{2d}}a(x,\omega)|\langle f,\varphi^{{{h}}}_{(x,\omega)}\rangle|^2\, dx\,d\omega<\infty\}.
\]
\begin{lemma}\label{lem lemma4.3} The set $ \mathcal{D}_{a,h}$ is a vector space that contains the Schwartz class $\cS(\bR^d)$. Moreover $\mathcal{D}_{a,h}$ is complete
 with respect to the norm \[
\|f\|^2_{a,h}:=(2\pi h)^{-d}\int_{\bR^{2d}}(a(x,\omega)+1)|\langle f,\varphi^{{{h}}}_{(x,\omega)}\rangle|^2\, dx\,d\omega.
\]
\end{lemma}
\begin{proof} It is clear that $\mathcal{D}_{a,h}$ is a vector space. If $f\in \cS(\bR^d)$ we have that the function $(x,\omega)\mapsto \langle f,\varphi^{{{h}}}_{(x,\omega)}\rangle$ is Schwartz; in particular, for every $M>0$ there exists a constant $C_M$ such that 
\[
|\langle f,\varphi^{{{h}}}_{(x,\omega)}\rangle|^2\leq C_M(1+|x|^2+|\omega|^2)^{-M/2}.
\]
 Setting for brevity $\Phi(x,\omega):=a(x,\omega)|\langle f,\varphi^{{{h}}}_{(x,\omega)}\rangle|^2$, we therefore have, for every $M>0$,
\begin{align*}
\int_{\bR^{2d}}&\Phi(x,\omega)\, dx\, d\omega=\int_{B(0,1)}\Phi(x,\omega)\, dx\, d\omega+\sum_{k=0}^\infty \int_{B(0,2^{k+1})\setminus B(0,2^k)}\Phi(x,\omega)\, dx\, d\omega\\
&\leq C_M\Big( \int_{B(0,1)} a(x,\omega)\, dx\, d\omega+\sum_{k=0}^\infty 2^{-kM} \int_{B(0,2^{k+1})\setminus B(0,2^k)}a(x,\omega)\, dx\, d\omega\Big)\\
&\leq C_M\Big(\int_{B(0,1)} a(x,\omega)\, dx\, d\omega+\sum_{k=0}^\infty 2^{-kM} \int_{B(0,2^{k+1})}a(x,\omega)\, dx\, d\omega\Big).
\end{align*}
This expression is certainly finite if we choose $M>N$, where $N$ is the constant in \eqref{eq polgr}. 

Concerning the second statement, let $f_n$ be a Cauchy sequence; hence, for every $\varepsilon>0$, 
\begin{equation}\label{eq Cauchy}
\int_{\bR^{2d}}(a(x,\omega)+1)|\langle f_n-f_m,\varphi^{{{h}}}_{(x,\omega)}\rangle|^2\, dx\,d\omega<\varepsilon
\end{equation}
for $n,m>N_\varepsilon$. Then by \eqref{eq parseval} we see that $f_n$ converges in $ L^2(\bR^d)$ to some $f$ and therefore $\langle f_n,\varphi^{{{h}}}_{(x,\omega)}\rangle\to \langle f ,\varphi^{{{h}}}_{(x,\omega)}\rangle$ pointwise. Letting $m\to+\infty$ in \eqref{eq Cauchy} and using the Fatou Lemma we obtain
\[
\int_{\bR^{2d}}(a(x,\omega)+1)|\langle f_n-f,\varphi^{{{h}}}_{(x,\omega)}\rangle|^2\, dx\,d\omega\leq\varepsilon
\]
for $n>N_\varepsilon$, namely $f\in \mathcal{D}_{a,h}$ and $\|f_n- f\|_{a,h}\to 0$.
\end{proof}

As a consequence of the previous lemma, the quadratic form 
\begin{equation}\label{eq qa0}
Q_{a,h}(f):=(2\pi{{h}})^{-d}\int_{\bR^{2d}}a(x,\omega)|\langle f,\varphi^{{{h}}}_{(x,\omega)}\rangle|^2\, dx\,d\omega,
\end{equation}
with domain $\mathcal{D}_{a,h}$, is a densely defined nonnegative closed form (notice that $\|f\|_{a,h}^2=Q_{a,h}(f)+\|f\|^2_{L^2}$). Hence it is canonically associated with a unique self-adjoint operator $A_{a,h}$ in $L^2(\bR^d)$ with domain ${\rm dom}\,A_{a,h}\subset \mathcal{D}_{a,h}$ and such that
\begin{equation}\label{eq domainq}
\langle A_{a,h} f,g\rangle=Q_{a,h}(f,g)\qquad f\in{\rm dom}\,A_{a,h},\ g\in \mathcal{D}_{a,h},
\end{equation} 
where $Q_{a,h}(\cdot,\cdot)$ is the sesquilinear form corresponding to the quadratic form $Q_{a,h}(\cdot)$ (see e.g. \cite[Theorem 1.18]{frank}). Moreover, denoting by $\sigma(A_{a,h})$ its spectrum, we have 
\begin{align}\label{eq qa}
\inf \sigma(A_{a,h})&=  \inf_{f\in\mathcal{D}_{a,h}\setminus\{0\}}\frac{Q_{a,h}(f)}{\|f\|_{L^2}^2}\nonumber \\
&= \inf_{f\in L^2(\bR^{d})\setminus\{0\}}\frac{(2\pi{{h}})^{-d}\int_{\bR^{2d}}a(x,\omega)|\langle f,\varphi^{{{h}}}_{(x,\omega)}\rangle|^2\, dx\,d\omega}{\|f\|_{L^2}^2},
\end{align}
where in the last equality we used the fact that if $f\in L^2(\bR^{d})\setminus \mathcal{D}_{a,h}$ the above integral is $+\infty$. 
\begin{remark}
We claim that $\mathcal{S}(\bR^d)\subset{\rm dom}\,A_{a,h}$ and $A_{a,h} f={\rm Op}^{\rm AW}_h(a)f$ for $f\in \mathcal{S}(\bR^d)$, where ${\rm Op}^{\rm AW}_h(a)$ is defined in Definition \ref{def antiwick}. 

Indeed, we know that the domain ${\rm dom}\,A_{a,h}$ consists of the functions $f\in\mathcal{D}_{a,h}$ such that 
\[
|Q_{a,h}(f,g)|\leq C\|g\|_{L^2}\quad\forall g\in \mathcal{D}_{a,h}
\]
for some $C>0$.  Since $|\langle g,\varphi^{{{h}}}_{(x,\omega)}\rangle|\leq \|g\|_{L^2}$ by the Cauchy-Schwarz inequality, we see that the latter estimate is certainly satisfied when the function $a(x,\omega)\langle f,\varphi^{{{h}}}_{(x,\omega)}\rangle$ belongs to $L^1(\bR^{2d})$. By arguing as in the proof of Lemma \ref{lem lemma4.3}, we see that this is the case if $f\in\mathcal{S}(\bR^d)$. Hence $\mathcal{S}(\bR^d)\subset{\rm dom}\,A_{a,h}$ and the claim then follows from \eqref{eq domainq}. 

Hence $A_{a,h}$ is a self-adjoint extension of ${\rm Op}^{\rm AW}_h(a)$, regarded as a densely defined symmetric operator in $L^2(\bR^d)$ with domain $\mathcal{S}(\bR^d)$. One could similarly consider any closed subspace $\tilde{D}\subset \mathcal{D}_{a,h}$ (with respect to the norm $\|\cdot\|_{a,h}$), with $\mathcal{S}(\bR^d)\subset \tilde{D}$, and consequently the self-adjoint operator $\tilde{A}$ associated with the quadratic form $Q_{a,h}$ with form domain $\tilde{D}$. In particular, if $\tilde{D}$ is the closure of $\mathcal{S}(\bR^d)$ in $\mathcal{D}_{a,h}$ the operator $\tilde{A}$ is the Friedrichs extension. In the subsequent sections we will provide lower and upper bounds for the bottom of the spectrum and of the essential spectrum of $A_{a,h}$, but the results will apply equally to these realizations $\tilde{A}$. This is true for the lower bounds, since the quadratic form associated with $A_{a,h}$ extends that associated with $\tilde{A}$, and, as regards the upper bounds, because only the values $Q_{a,h}(f)$ with $f\in \mathcal{S}(\bR^d)$ will enter in their proofs.

 \end{remark}
 \begin{remark}
 The space $\mathcal{D}_{a,h}$, endowed with the norm $\|\cdot\|_{a,h}$ (and $h=1$), coincides with the {\it weighted modulation space} denoted in the literature by $M^{2,2}_{a+1}$; see \cite[Chapter 11]{grochenig-book}. However, in the theory of modulation spaces the weight (the function $1+a$, in this case) is usually required to satisfy certain further conditions (see \cite[Definition 11.1.1]{grochenig-book}), whereas here we only assume that $a$ is nonnegative and satisfies the growth condition \eqref{eq polgr}. Under those extra assumptions it follows from \cite[Proposition 11.3.4]{grochenig-book} that $\mathcal{S}(\bR^{d})$ is in fact dense in $\mathcal{D}_{a,h}$. However we will not need this property. 
 \end{remark}
\section{The bottom of the spectrum - Muckenhoupt symbols}
We use the same notation as in Section \ref{sec quadratic form}. We are interested in estimating the bottom of the spectrum and of the essential spectrum of the operator $A_{a,h}$ (defined in Section \ref{sec quadratic form} via the quadratic form $Q_{a,h}$ in \eqref{eq qa} with form domain $\mathcal{D}_{a,h}$ and a window $\varphi\in\mathcal{S}(\bR^{d})$ with $\|\varphi\|_{L^2}=1$) in terms of the nonnegative symbol $a$. Here we will consider symbols in the Muckenhoupt class $A_\infty(\bR^{2d})$. 
We will use \eqref{eq qa} and Theorem \ref{teomain} applied to the Schr\"odinger representation $\pi$ of the reduced Heisenberg group $\mathbb{H}^d$, as defined in Section \ref{sec heisenberg}. 
\subsection{The bottom of the spectrum}
We observe that, by rescaling, for $f\in L^2(\bR^d)$, 
\begin{align}\label{eq resc}
  (2\pi{{h}})^{-d}\int_{\bR^{2d}}&a(x,\omega)|\langle f,\varphi^{{{h}}}_{(x,\omega)}\rangle|^2\, dx\,d\omega\nonumber \\\
  &=(2\pi)^{-d}\int_{\bR^{2d}}a({h}^{1/2} x,{h}^{1/2}\omega)| \langle D_{{h}} f,\varphi_{(x,\omega)}\rangle|^2\, dx\,d\omega,
\end{align}
where we have set $D_{{h}} f(x):={{h}}^{d/4}f({{h}}^{1/2} x)$ for the unitary dilation and 
\begin{equation}\label{eq phixo}
\varphi_{(x_0,\omega_0)}(x):=e^{i\omega_0 x} \varphi(x-x_0) \qquad x,x_0,\omega_0\in\bR^d,
\end{equation}
i.e. $\varphi_{(x_0,\omega_0)}=\varphi^{{h}}_{(x_0,\omega_0)}$ with ${{h}}=1$. Also observe that for the representation $\pi$ in \eqref{eq pih} we have $\pi(x,\omega,t)\varphi=e^{it}\varphi_{(x,\omega)}$. 
Hence, by \eqref{eq qa} we can write
\begin{align}\label{eq reduc}
\inf\sigma(A_{a,h})&
\\
=\inf_{f\in  L^2(\bR^d)\setminus\{0\}}& \frac{(2\pi)^{-d}\int_{\mathbb{R}^{2d}}a({h}^{1/2}x,{h}^{1/2}\omega)|\langle f,\varphi_{(x,\omega)}\rangle |^2\,dx\,d\omega}{\|f\|_{L^2}^2}\nonumber\\
= \inf_{f\in  L^2(\bR^d)\setminus\{0\}}& \frac{(2\pi)^{-(d+1)}\int_{\mathbb{H}^{d}}a({h}^{1/2}x,{h}^{1/2}\omega)|\langle f,\pi(x,\omega,t)\varphi\rangle |^2\,dx\,d\omega\,dt}{\|f\|_{L^2}^2}.\nonumber
 \end{align}

  Also, let
\begin{align}\label{eq lambda} 
\boldsymbol{\lambda}(a,h):&=\inf\left\{\fint_B a(x,\omega)\, dx\,d\omega:\ B\textrm{ is a ball of radius }{h}^{1/2}\right\} \\
\label{eq resc2} 
&=\inf\left\{\fint_B a({h}^{1/2}x,{h}^{1/2}\omega)\, dx\,d\omega:\ B\textrm{ is a ball of radius }1\right\}. 
\end{align}

\begin{theorem}\label{mainteo2}
Let $\alpha\in(0,1)$. There exists $C=C(\alpha,\varphi)>0$ (depending on the dimension $d$ as well) such that, for every $a\in A^\alpha_{\infty}(\bR^{2d})$ and ${{h}}>0$, 
\begin{equation}\label{eq ineq}
C^{-1} \boldsymbol{\lambda}(a,h)\leq \inf \sigma(A_{a,h})\leq C \boldsymbol{\lambda}(a,h).
\end{equation}
\end{theorem} 
\begin{proof}[Proof of Theorem \ref{mainteo2}]
First of all we notice that $a$ satisfies \eqref{eq polgr} for some $C,N>0$, because the measure $a(x,\omega)\, dx\, d\omega$ is doubling. More precisely, we observe for future reference (see e.g.  \cite[Theorem 9.3.3]{grafakos}) that there exists a constant $C>0$ depending on $\alpha$ and $d$ such that for every $a\in A_\infty^\alpha(\bR^{2d})$ and $x_0,\omega_0\in\bR^d$, $k\geq 0$,
\begin{equation}\label{eq doub inf}
\int_{B((x_0,\omega_0),2^k)}a(x,\omega)\, dx\, d\omega \leq C^{k} \int_{B((x_0,\omega_0),1)}a(x,\omega)\, dx\, d\omega.
\end{equation}

 Also, observe that the function $a({h}^{1/2}x,{h}^{1/2}\omega)$ belongs to $A_{\infty}^\alpha (\bR^{2d})$ for every ${{h}}>0$, by the dilation invariance of this class. 

The first inequality in \eqref{eq ineq} then follows by applying \eqref{eq reduc} and Theorem \ref{teomain} with $p=2$ where: 
\begin{itemize}
\item{} $\mathcal{H}=L^2(\bR^d)$, $G=\mathbb{H}^d$ is the reduced Heisenberg group, and $\pi$ is the Schr\"odinger representation of $\mathbb{H}^d$ (cf. Section \ref{sec heisenberg}).
\item $\phi=\varphi$.
\item $\chi=1$ and $\mu_\chi$ is therefore the Haar measure of $\mathbb{H}^d$.
\item The left-invariant Riemannian metric on $\mathbb{H}^d$ is induced by the vector fields $X_j$, $\Omega_j$, $j=1,\ldots,d$ and $T$ in Section \ref{sec heisenberg}.
\item The weight $w$ on $\mathbb{H}^d$ is given by $w(x,\omega,t)=a({h}^{1/2}x,{h}^{1/2}\omega)$.
\end{itemize} 
Indeed, since $a\in A^\alpha_\infty(\bR^{2d})\subset  A^\alpha_{\infty,{\rm loc}}(\bR^{2d})$, by Lemma \ref{lem equiv} we have  $w\in A^\beta_{\infty,{\rm loc}}(\mathbb{H}^d)$, for some $\beta\in(0,1)$ depending only on $\alpha$ and $d$. Also, \eqref{eq phi} is satisfied because 
\[
V_\phi\phi(x,\omega,t)=(2\pi)^{-(d+1)/2}\langle \varphi, \pi(x,\omega,t)\varphi\rangle=(2\pi)^{-(d+1)/2}e^{-it}\langle \varphi,\varphi_{(x,\omega)}\rangle,
\]
the function $\langle \varphi,\varphi_{(x,\omega)}\rangle$ is Schwartz, and $X_j,\Omega_j$ and $T$ have polynomial coefficients. 
 
 The last sentence in Lemma \ref{lem equiv} and \eqref{eq resc2} allow us to transfer to $\bR^{2d}$ the lower bound on $\mathbb{H}^d$ obtained from Theorem \ref{teomain}. 
 
Concerning the second inequality in \eqref{eq ineq}, by \eqref{eq reduc} and \eqref{eq resc2} we see that it suffices to prove that, given $\alpha\in(0,1)$, there exists a constant $C=C(\alpha,\varphi)>0$ such that for every $x_0,\omega_0\in\bR^d$, $a\in A^\alpha_\infty(\bR^{2d})$ and ${{h}}>0$,
\begin{align}\label{eq 26gen}
\int_{\bR^{2d}}a({h}^{1/2}x,{h}^{1/2}\omega)|\langle \varphi_{(x_0,\omega_0)},&\varphi_{(x,\omega)}\rangle |^2\,dx\,d\omega\\
&\leq C\int_{B((x_0,\omega_0),1)}
a({h}^{1/2}x,{h}^{1/2}\omega)\,dx\,d\omega.\nonumber
\end{align}
Namely, it is enough to prove that for every $a\in A^\alpha_\infty(\bR^{2d})$
\begin{align}\label{eq 26genbis}
\int_{\bR^{2d}}a(x,\omega)|\langle \varphi,\varphi_{(x-x_0,\omega-\omega_0)}&\rangle |^2\,dx\,d\omega\nonumber \\
&\leq C\int_{B((x_0,\omega_0),1)}
a(x,\omega)\,dx\,d\omega
\end{align}
with a constant $C$ depending only on $\alpha$ and $\varphi$.
 
Since the expression $\langle \varphi,\varphi_{(x,\omega)}\rangle$ is a Schwartz function, the latter estimate follows as in the proof of Lemma \ref{lem lemma4.3} (with ${{h}}=1$), considering annuli centered at $(x_0,\omega_0)$, and using \eqref{eq doub inf} to conclude.  
\end{proof}
\subsection{The bottom of the essential spectrum}
We now estimate the bottom of the essential spectrum $\sigma_{\rm ess}(A_{a,h})$. We recall its characterization in terms of the associated quadratic form (cf. \cite[Lemma 1.20]{frank}): 
\begin{align}\label{eq ess sp}
\inf \sigma_{\rm ess}(A_{a,h})=\inf\{\liminf_{n\to\infty}& \,Q_{a,h}(f_n):\ (f_n)\subset\mathcal{D}_{a,h},\\
&\ \|f_n\|_{L^2}=1,\ f_n\to 0 \textrm{ weakly in } L^2(\bR^d)\}.\nonumber
\end{align}
Let 
\begin{align*}
\boldsymbol{\lambda}_{\rm ess}(a,h):&=\liminf_{(x_0,\omega_0)\to\infty} \fint_{B((x_0,\omega_0),{h}^{1/2})} a(x,\omega)\, dx\,d\omega\\
&=\liminf_{(x_0,\omega_0)\to\infty} \fint_{B((x_0,\omega_0),1)} a({h}^{1/2}x,{h}^{1/2}\omega)\, dx\,d\omega.
\end{align*}
\begin{corollary}\label{maincor}
Let $\alpha\in(0,1)$. There exists $C(\alpha,\varphi)>0$ such that, for every $a\in A^\alpha_{\infty}(\bR^{2d})$ and ${{h}}>0$, 
\begin{equation}\label{eq essential}
C^{-1} \boldsymbol{\lambda}_{\rm ess}(a,h)\leq \inf \sigma_{\rm ess}(A_{a,h})\leq C \boldsymbol{\lambda}_{\rm ess}(a,h).
\end{equation}
In particular, $\sigma(A_{a,h})$ is purely discrete if and only if 
\[
\lim_{(x_0,\omega_0)\to\infty}\int_{B((x_0,\omega_0),{h}^{1/2})} a(x,\omega)\, dx\,d\omega=+\infty.
\]
\end{corollary}
\begin{proof}
Let us prove the first inequality in \eqref{eq essential}. We first suppose $\boldsymbol{\lambda}_{\rm ess}(a,h)<+\infty$ (which holds for some $h>0$ if and only if it holds for every $h>0$, because the measure $a(x,\omega)\, dx\, d\omega$ is doubling). 
Let $R>0$ such that 
\begin{equation}\label{eq 1aa}
\fint_{B((x_0,\omega_0),{h}^{1/2})} a(x,\omega)\, dx\,d\omega\geq \frac{1}{2}\boldsymbol{\lambda}_{\rm ess}(a,h)
\end{equation}
for $|x_0|^2+|\omega_0|^2>R^2$. Consider the function 
\[
b(x,\omega):= C(1+|x|+|\omega|)^{-1},
\]  
where the constant $C>0$ is chosen so that 
\begin{equation}\label{eq 2aa} 
b(x,\omega)\geq  \frac{1}{2}\boldsymbol{\lambda}_{\rm ess}(a,h)\qquad{\rm for}\ |x|^2+|\omega|^2\leq (R+{h}^{1/2})^2.  
\end{equation}
Observe that $C$ will depend on $a$ and ${{h}}$ (and $R$, which here is fixed), but $b\in A^\beta_\infty(\bR^{2d})$ for some $\beta\in (0,1)$ depending only on $d$. Indeed, the weight $|z|^{-1}$, $z=(x,\omega)\in\bR^{2d}$ belongs to $A_\infty(\bR^{2d})$ (see e.g. \cite[page 218]{stein_book}) and therefore the same holds for $(1+|x|+|\omega|)^{-1}\asymp (1+|z|)^{-1} \asymp\min\{1,|z|^{-1}\}$ because $A_\infty(\bR^{2d})$ is a lattice; see e.g. \cite[Proposition 4.3]{kilpelainen}). Hence $a+b\in A^{\min\{\alpha,\beta\}}_\infty(\bR^{2d})$ and $\mathcal{D}_{a+b,h}=\mathcal{D}_{a,h}$ because $b$ is bounded.

We will check in a moment that
\begin{equation}\label{eq pert}
\sigma_{\rm ess}(A_{a,h})= \sigma_{\rm ess}(A_{a+b,h}).
\end{equation}
Applying Theorem \ref{mainteo2} to $A_{a+b,h}$ we see that, for $f\in \mathcal{D}_{a+b,h}=\mathcal{D}_{a,h}$ and $\|f\|_{L^2}=1$, 
\begin{equation}\label{eq pert2}
Q_{a+b,h}(f) \geq C \boldsymbol{\lambda}(a+b,h),
\end{equation}
 cf. \eqref{eq lambda}, for a constant $C$ depending only on $\alpha$ and $\varphi$. 
 
 On the other hand we have 
 \begin{align*}
 &\boldsymbol{\lambda}(a+b,h)\\
 &=\inf_{(x_0,\omega_0)\in\bR^{2d}}\Big(\fint_{B((x_0,\omega_0),{h}^{1/2})}a(x,\omega)\, dx\,d\omega+\fint_{B((x_0,\omega_0),{h}^{1/2})}b(x,\omega)\, dx\,d\omega\Big)\\
 &\geq\frac{1}{2}\boldsymbol{\lambda}_{\rm ess}(a,h),
 \end{align*}
 by \eqref{eq 1aa} and \eqref{eq 2aa}. Hence, by \eqref{eq pert}, \eqref{eq ess sp} (with $a+b$ in place of $a$) and \eqref{eq qa} (since $\inf \sigma(A_{a+b,h})\leq \inf \sigma_{\rm ess}(A_{a+b,h})$) we obtain the first inequality of \eqref{eq essential}. 
 
 It remains to check \eqref{eq pert}. This follows from Weyl's theorem for quadratic forms (cf. \cite[Theorem 1.51]{frank}) if we prove that the (nonnegative) quadratic form $Q_{b,h}$ (cf. \eqref{eq qa0})
is compact in $\mathcal{D}_{a,h}$, i.e. that for some $C>0$,
 \begin{equation}\label{eq qb}
 Q_{b,h}(f)\leq C\|f\|^2_{a,h} \qquad \forall f\in \mathcal{D}_{a,h},
 \end{equation}
and that the bounded operator in $\mathcal{D}_{a,h}$ induced by the form $Q_{b,h}$ via the Riesz representation theorem is compact. Now, since $b$ is bounded, the quadratic form $Q_{b,h}$ comes in fact from a bounded operator in $L^2(\bR^{d})$ and it suffices to prove that this operator is compact in $L^2(\bR^d)$ (cf. \cite[page 56]{frank}). But this is well known, since $b\to0$ at infinity (cf. \cite[Theorem 14.5]{wong}). 

The case $\boldsymbol{\lambda}_{\rm ess}(a,h)=+\infty$ requires only minor modifications. That is, for every $t>0$ we can find $R$ and $C$ such that  \eqref{eq 1aa} and \eqref{eq 2aa} hold with $\frac{1}{2}\boldsymbol{\lambda}_{\rm ess}(a,h)$ replaced by $t$. Then we can repeat the same argument as above and we arrive at $t\leq \inf \sigma(A_{a+b,h})\leq \inf \sigma_{\rm ess}(A_{a+b,h})=\inf \sigma_{\rm ess}(A_{a,h})$. Since $t$ is arbitrary, we obtain $\inf \sigma_{\rm ess}(A_{a,h})=+\infty$.

Let us now prove the second inequality in \eqref{eq essential}. By \eqref{eq ess sp}, \eqref{eq qa0} and \eqref{eq resc} we see that it suffices to prove that there exists $C>0$ depending on $\alpha$ and $\varphi$ such that, for every $a\in A^\alpha_\infty(\bR^{2d})$, ${{h}}>0$ and {\it every} sequence $(x_n,\omega_n)\to\infty$, we have 
\[
\int_{\bR^{2d}}a({h}^{1/2} x,{h}^{1/2}\omega)| \langle f_n,\varphi_{(x,\omega)}\rangle|^2\, dx\,d\omega\leq C \int_{B((x_n,\omega_n),1)} a({h}^{1/2} x,{h}^{1/2}\omega)\, dx\, d\omega
\]
for {\it some} sequence $f_n\in\mathcal{S}(\bR^d)$, $\|f_n\|_{L^2}=1$, with $f_n\to0$ weakly in $L^2(\bR^d)$. Now, this is true with $f_n=\varphi_{(x_n,\omega_n)}$ by \eqref{eq 26gen} and, moreover, $f_n\to0$ weakly in $L^2(\bR^d)$. Indeed, for $g\in L^2(\bR^d)$, $\langle g, \varphi_{(x,\omega)} \rangle\to 0$  as $(x,\omega)\to\infty$ because this holds if $g\in \mathcal{S}(\bR^{d})$ ($\langle g, \varphi_{(x,\omega)} \rangle$ is a Schwartz function, in that case) and we have the uniform bound $|\langle g, \varphi_{(x,\omega)} \rangle|\leq \|g\|_{L^2}$ by the Cauchy-Schwarz inequality. 
\end{proof}
\begin{corollary}\label{cor poli}
For every polynomial $P(x,\omega)$ in $\bR^{2d}$ of degree $N\geq 1$, and $\beta>-1/N$, the conclusions of Theorem \ref{mainteo2} and Corollary \ref{maincor} hold true for $a=|P|^\beta$, with a constant $C$ depending only on $N, \beta, d$ and $\varphi$. 
\end{corollary}
\begin{proof}
This follows by observing that such a function $a$ belongs to $A^\alpha_\infty(\bR^{2d})$, for some $\alpha\in (0,1)$ depending only on $N$, $d$ and $\beta$ (see \cite[page 219]{stein_book}). 
\end{proof}
\begin{corollary}\label{cor 7mar}\ 
For every $\alpha\in(0,1)$ there exists $C=C(\alpha,\varphi)>0$ such that the following holds true. For every measurable function $a(x,\omega)\geq0$, and $h>0$, setting $\gamma_0:={\rm ess\,inf}_{\bR^{2d}}\,a$, we have 
\[
C^{-1}\boldsymbol{\lambda}(a-\gamma_0,h)\leq \inf\sigma(A_a)-\gamma_0\leq C\boldsymbol{\lambda}(a-\gamma_0,h),
\]
provided $a-\gamma_0\in A^\alpha_\infty(\bR^{2d})$. 
\end{corollary}
\begin{proof}
This follows by applying Theorem \ref{mainteo2} to the symbol $a-\gamma_0$. 
\end{proof}
\section{The bottom of the spectrum --- semiclassical symbols}\label{sec semiclassicalbis} In this section we continue the study of the bottom of the spectrum of the anti-Wick operator $A_{a,h}$ with a nonnegative symbol $a$ (see Section \ref{sec quadratic form} for the notation), that is here assumed to belong to the symbol classes appearing in semiclassical analysis (cf. \cite[page 81]{dimassi}, \cite[Definition A.2.2]{shubin_book} and \cite[Section 4.3]{zworski}). We moreover suppose $h\in (0,1]$. 
\begin{definition} For $m\in\mathbb{R}$, $\rho\in [0,1/2)$, we denote by $S^m_\rho$ the space of the families of smooth functions $a(\cdot,h)$ in $\bR^{2d}$, depending on a parameter $h\in (0,1]$ such that for some $N_0\in\mathbb{N}$ and every $\alpha,\beta\in\mathbb{N}^d$, with $|\alpha|+|\beta|\geq N_0$,
\begin{equation}\label{eq asemi}
\sup_{(x,\omega,h)\in\bR^{2d}\times(0,1]}|\partial^\alpha_x\partial^\beta_\omega a(x,\omega,h)|h^{m+\rho(|\alpha|+|\beta|)}<\infty.
\end{equation}
\end{definition}
We observe that a finite number of derivatives $\partial^\alpha_x\partial^\beta_\omega a$ are allowed to be unbounded. 

Notice that a nonnegative symbol $a\in S^m_\rho$ certainly satisfies the growth condition \eqref{eq polgr}. Hence, as in Section \ref{sec quadratic form} we can consider the corresponding anti-Wick operator $A_{a,h}$ defined via the  quadratic form $Q_{a,h}$ in \eqref{eq qa0}, for a window $\varphi\in \mathcal{S}(\bR^{d})$, with $\|\varphi\|_{L^2}=1$. 
Observe that $A_{a,h}$ will therefore depend on $h$ even through the symbol $a$.

Let 
\begin{equation}\label{eq lambdaprime}
\boldsymbol{\lambda}'(a,h):=\inf_{(x_0,\omega_0)\in\mathbb{R}^{2d}}\sup_{(x,\omega)\in B((x_0,\omega_0),{h}^{1/2})}a(x,\omega,h).
\end{equation}

\begin{theorem}\label{teo semi}
Let $a\in S^m_\rho$ for some $m\in\mathbb{R}$, $\rho\in [0,1/2)$, with $a\geq0$. 

For every $N\in\mathbb{N}$ there exists $C_N>0$ such that, for every $h\in (0,1]$,
\begin{equation}\label{eq ulower}
C_N^{-1}\boldsymbol{\lambda}'(a,h)-C_N h^{N}\leq \inf \sigma(A_{a,h})\leq C_N(\boldsymbol{\lambda}'(a,h)+h^{N}). 
\end{equation}
\end{theorem}
\begin{proof} Let us prove the first inequality in \eqref{eq ulower}. By \eqref{eq reduc} we have to prove that, for $f\in L^2(\bR^{d})$, $\|f\|_{L^2}=1$, 
\begin{equation}\label{eq utop}
\int_{\bR^{2d}}a({h}^{1/2} x,{h}^{1/2}\omega,h)| \langle f,\varphi_{(x,\omega)}\rangle|^2\, dx\,d\omega\geq C_N^{-1}\boldsymbol{\lambda}'(a,h)-C_N h^{N}, 
\end{equation}
where the functions $\varphi_{(x,\omega)}$ are defined in \eqref{eq phixo}.

We know that $a$ satisfies \eqref{eq asemi} for $|\alpha|+|\beta|\geq N_0$, for some $N_0\in\mathbb{N}$. For $N\geq0$,  let $M=M_N\geq0$ be the minimum natural number such that 
\[
\begin{cases}
M+1\geq N_0\\
-m+(\frac{1}{2}-\rho)(M+1)\geq N.
\end{cases}
\]
Then 
\[
|\partial^\alpha_x\partial^\beta_\omega (a({h}^{1/2}x,{h}^{1/2}\omega,h))|\leq C'_N h^N\quad{\rm for}\ |\alpha|+|\beta|=M+1. 
\]
Let $P_{(x_0,\omega_0),M}$ be the Taylor polynomial of $a({h}^{1/2}x,{h}^{1/2}\omega,h)$ of order $M=M_N$ centered at $(x_0,\omega_0)\in\bR^{2d}$. We have 
\begin{equation}\label{eq ueqap}
|a({h}^{1/2}x,{h}^{1/2}\omega,h)-P_{(x_0,\omega_0),M}(x,\omega,h)|\leq C''_N h^N,\quad (x,\omega)\in B((x_0,\omega_0),1)
\end{equation}
for a constant $C''_N$ depending on $N$. 

Now, for $\varepsilon>0$ and $r\in(0,1]$ let
\[
\Gamma_{\varepsilon,r,N}:=\varepsilon+\inf_{(x_0,\omega_0,t_0)\in\mathbb{H}^d}\fint_{B((x_0,\omega_0,t_0),r)} |P_{(x_0,\omega_0),M}(x,\omega,h)|\, dx\, d\omega\, dt,
\]
where $B((x_0,\omega_0,t_0),r)$ is a ball in the reduced Heisenberg group $\mathbb{H}^d$ (see Section \ref{sec heisenberg} for the relevant definitions). The constant $\varepsilon$ is introduced to avoid potential division by zero later. Clearly $\Gamma_{\varepsilon,r,N}\geq\varepsilon$. 

For $\varepsilon>0$, $(x_0,\omega_0,t_0)\in\mathbb{H}^d$ and $r\in (0,1]$, consider the weight on $\mathbb{H}^d$, constant on the fibers, given by
\[
w_{\varepsilon,r,N,x_0,\omega_0}(x,\omega,t):= r^{-2} \Gamma_{\varepsilon,r,N}^{-1} (\varepsilon+|P_{(x_0,\omega_0),M}(x,\omega,h)|).
\]

Let us observe that the function $|P_{(x_0,\omega_0),M}|$, if is not identically zero, belongs to the class $A^\alpha_\infty(\bR^{2d})\subset A^\alpha_{\infty,{\rm loc}}(\bR^{2d})$ for some $\alpha\in(0,1)$ depending only on $M$, i.e. on $N$ (see e.g. \cite[page 219]{stein_book}). The same holds, for the same $\alpha$, for the functions $r^{-2} \Gamma_{\varepsilon,r,N}^{-1} (\varepsilon+|P_{(x_0,\omega_0),M}|)$ (we can suppose $\alpha\leq 1/2$, so that the constant function $\varepsilon$ belongs to $A^\alpha_{\infty}(\bR^{2d})$; moreover this holds even if $P_{(x_0,\omega_0),M}$ is identically zero). Hence, by Lemma \ref{lem equiv}, the functions $w_{\varepsilon,r,N,x_0,\omega_0}$ belong to $A^\beta_{\infty,{\rm loc}}(\mathbb{H}^d)$ for some $\beta\in(0,1)$ depending only on $N$. Moreover, by the very definition of $\Gamma_{\varepsilon,r,N}$,
\[
\fint_{B((x_0,\omega_0,t_0),r)} w_{\varepsilon,r,N,x_0,\omega_0}(x,\omega,t)\,dx\,d\omega\, dt\geq r^{-2}. 
\]
As a consequence of Lemma \ref{mainlemma} applied to $\mathbb{H}^d$, and with $p=2$, denoting by $\nabla$ the gradient with respect to the basis of left-invariant vector fields on $\mathbb{H}^d$ as in Section \ref{sec heisenberg}, we have, for $r\in (0,1]$,
\begin{align}\label{eq ueq1}
\int_{B((x_0,\omega_0,t_0),r)} &\big(|\nabla u(x,\omega,t)|^2+ w_{\varepsilon,r,N,x_0,\omega_0}(x,\omega,t)| u(x,\omega,t)|^2\big)dx\,d\omega\,dt\nonumber \\
&\geq C'''_Nr^{-2}\int_{B((x_0,\omega_0,t_0),r)}| u(x,\omega,t)|^2dx\,d\omega\,dt
\end{align}
for every $u\in L^2_{\rm loc}(\mathbb{H}^{d})$, with $\nabla u\in L^2_{\rm loc}(\mathbb{H}^{d})$. 

By Remark \ref{rem covar} (a) there exists $r_0\in(0,1]$ such that the projection on $\bR^{2d}$ of every ball $B((x_0,\omega_0,t_0),r_0)\subset\mathbb{H}^d$ is contained in the Euclidean ball $B((x_0,\omega_0),1)$. Hence, by \eqref{eq ueqap}, if $r\in (0,r_0]$ we have, on $B((x_0,\omega_0,t_0),r)$ 
\[
|P_{(x_0,\omega_0),M}(x,\omega,h)|\leq a({h}^{1/2}x,{h}^{1/2}\omega,h)+C''_Nh^N.
\]
Therefore, if $r\in(0,r_0]$,
\begin{align}\label{eq ueq2}
\int_{B((x_0,\omega_0,t_0),r)} &\Big(|\nabla u(x,\omega,t)|^2\nonumber\\
&+ \frac{a({h}^{1/2}x,{h}^{1/2}\omega,h)+C''_Nh^N+\varepsilon}{r^2\Gamma_{\varepsilon,r,N}} | u(x,\omega,t)|^2\Big)dx\,d\omega\,dt\nonumber \\
&\geq C'''_Nr^{-2}\int_{B((x_0,\omega_0,t_0),r)}| u(x,\omega,t)|^2dx\,d\omega\,dt.
\end{align}

Arguing as in the proof of Theorem \ref{teomain}, with $u(x,\omega,t)=\langle f,\pi(x,\omega,t)\varphi\rangle=e^{-it}\langle f,\varphi_{(x,\omega)}\rangle$ (where $\pi$ is the Schr\"odinger  representation of $\mathbb{H}^d$; cf. Section \ref{sec heisenberg}) we obtain, for some $r\in(0,r_0]$ (small enough, depending on $\varphi$),
\begin{align*}
\int_{\bR^{2d}}\frac{a({h}^{1/2}x,{h}^{1/2}\omega,h)+C''_Nh^N+\varepsilon}{\Gamma_{\varepsilon,r,N}} &|\langle f,\varphi_{(x,\omega)}\rangle|^2dx\,d\omega\\
&\geq C''''_{N}\int_{\bR^{2d}}| \langle f,\varphi_{(x,\omega)}\rangle|^2dx\,d\omega.
\end{align*}
Since $\varepsilon$ was arbitrary, using \eqref{eq parseval} (with $h=1$) we see that this implies \eqref{eq utop} if we prove that, for some constant $c_{r,N}>0$,
\begin{equation}\label{eq ult}
\Gamma_{\varepsilon,r,N}\geq c_{r,N}(\boldsymbol{\lambda}'(a,h)-C''_N h^N), \quad r\in (0,1],\ h\in(0,1].
\end{equation} 
To this end, observe that, by Lemma \ref{lem equiv},
\begin{align*}
\fint_{B((x_0,\omega_0,t_0),r)}& |P_{(x_0,\omega_0),M}(x,\omega,h)|\, dx\, d\omega\, dt\\
&\geq c_{r,N}\fint_{B((x_0,\omega_0),r)} |P_{(x_0,\omega_0),M}(x,\omega,h)|\, dx\, d\omega\\
&\geq c'_{r,N}\sup_{(x,\omega)\in B((x_0,\omega_0),1)} |P_{(x_0,\omega_0),M}(x,\omega,h)|,
\end{align*}
where in the last inequality we used the fact that on the finite dimensional space of polynomials of degree $\leq M$ all norms are equivalent. 

On the other hand, by \eqref{eq ueqap} we have, for $(x,\omega)\in B((x_0,\omega_0),1)$,
\[
|P_{(x_0,\omega_0),M}(x,\omega,h)|\geq a({h}^{1/2}x,{h}^{1/2}\omega,h)-C''_Nh^{N}.
\]
Hence 
\begin{align*}
\fint_{B((x_0,\omega_0,t_0),r)}& |P_{(x_0,\omega_0),M}(x,\omega,h)|\, dx\, d\omega\, dt\\
&\geq c'_{r,N}\Big(\sup_{(x,\omega)\in B((x_0,\omega_0),1)} a({h}^{1/2}x,{h}^{1/2}\omega,h) -C''_Nh^{N}\Big).
\end{align*}
Taking the infimum over $(x_0,\omega_0,t_0)\in\mathbb{H}^d$ gives \eqref{eq ult} and therefore concludes the proof of the first inequality in \eqref{eq ulower}.

Finally let us prove the second inequality in \eqref{eq ulower}.  By \eqref{eq reduc} we see that it suffices  to show that, for every $(x_0,\omega_0)\in\mathbb{R}^{2d}$, $h\in(0,1]$,
\begin{align}\label{eq confron}
\int_{\bR^{2d}}a({h}^{1/2} x,{h}^{1/2}\omega,h)|& \langle \varphi_{(x_0,\omega_0)},\varphi_{(x,\omega)}\rangle|^2\, dx\,d\omega\nonumber\\
&\leq C_N \sup_{(x,\omega)\in B((x_0,\omega_0),1)} a({h}^{1/2} x,{h}^{1/2}\omega,h) +C_N h^{N}, 
\end{align}
where the functions $\varphi_{(x,\omega)}$ are defined in \eqref{eq phixo}.

With the same notation as above, by \eqref{eq asemi} we have 
\[
|a({h}^{1/2} x,{h}^{1/2}\omega,h)-P_{(x_0,\omega_0),M}(x,\omega,h)|\leq C'''''_N h^N (|x-x_0|+|\omega-\omega_0|)^{M+1}.
\]
Since the function $|\langle \varphi_{(x_0,\omega_0)},\varphi_{(x,\omega)}\rangle|=|\langle \varphi,\varphi_{(x-x_0,\omega-\omega_0)}\rangle|$ in \eqref{eq confron} has rapid decay in $x-x_0,\ \omega-\omega_0$, we see that it suffices to verify \eqref{eq confron} with the function $a({h}^{1/2} x,{h}^{1/2}\omega,h)$ replaced by $|P_{(x_0,\omega_0),M}(x,\omega,h)|$. The desired conclusion is then clear by \eqref{eq 26genbis}.  
\end{proof}
Arguing as in the proof of Corollary \ref{maincor} we also obtain the following result concerning the essential spectrum.

Let \[
\boldsymbol{\lambda}'_{\rm ess}(a,h):=\liminf_{(x_0,\omega_0)\to\infty}\sup_{(x,\omega)\in B((x_0,\omega_0),{h}^{1/2})}a(x,\omega,h).
\]
\begin{corollary}\label{cor semi}
Let $a\in S^m_\rho$ for some $m\in\mathbb{R}$, $\rho\in [0,1/2)$, with $a\geq0$. 
For every $N\in\mathbb{N}$ there exists $C_N>0$ such that, for every $h\in (0,1]$,
\begin{equation}\label{eq 26una}
C^{-1}_N\boldsymbol{\lambda}'_{\rm ess}(a,h)-C_N h^N\leq \inf \sigma_{\rm ess}(A_{a,h})\leq C_N(\boldsymbol{\lambda}'_{\rm ess}(a,h)+h^N).
\end{equation}
In particular, $\sigma(A_{a,h})$ is purely discrete if and only if 
\[
\lim_{(x_0,\omega_0)\to\infty}\sup_{(x,\omega)\in B((x_0,\omega_0),{h}^{1/2})}a(x,\omega,h)=+\infty.
\]

\end{corollary}
\begin{proof}
The proof is similar to that of Corollary \ref{maincor}, so that we only focus on the needed modifications. 

Let us prove the first inequality in \eqref{eq 26una} for all $h\in(0,1]$ such that $\boldsymbol{\lambda}'_{\rm ess}(a,h)<+\infty$. An easy modification of the argument below will show that the estimate holds for the $h$'s such that $\boldsymbol{\lambda}'_{\rm ess}(a,h)=+\infty$ as well. 

Similarly to the proof of Corollary \ref{maincor} one considers 
$R>0$ (depending on $a$ and $h$) such that 
\[
\sup_{(x,\omega)\in B((x_0,\omega_0),{h}^{1/2})} a(x,\omega,h)\geq \frac{1}{2}\boldsymbol{\lambda}'_{\rm ess}(a,h)
\]
for $|x_0|^2+|\omega_0|^2>R^2$, and threrefore the symbol $b(x,\omega,h):=C(1+|x|+|\omega|)^{-1}$ where the constant $C>0$ is chosen (depending on $a$ and $h$) so that 
\[ 
b(x,\omega,h)\geq  \frac{1}{2}\boldsymbol{\lambda}'_{\rm ess}(a,h)\qquad{\rm for}\ |x|^2+|\omega|^2\leq (R+{h}^{1/2})^2.  
\]
Hence 
\[
\boldsymbol{\lambda}'(a+b,h)\geq \frac{1}{2}\boldsymbol{\lambda}'_{\rm ess}(a,h),
\]
where the function $\boldsymbol{\lambda}'(\cdot,\cdot)$ was defined in \eqref{eq lambdaprime}.
Hence, the first inequality in \eqref{eq 26una} will follow as in the proof of Corollary \ref{maincor} if we prove the lower bound
\begin{equation}\label{eq 24}
Q_{a+b,h}(f)\geq C_N^{-1}\boldsymbol{\lambda}'(a+b,h)-C_N h^N,
\end{equation}
for all $f\in \mathcal{D}_{a+b,h}=\mathcal{D}_{a,h}$ with $\|f\|_{L^2}=1$, and $h\in(0,1]$.

This can be obtained by following a pattern similar to the proof of Theorem \ref{teo semi}. Namely, one still considers the Taylor polynomial $P_{(x_0,\omega_0),M}(x,\omega,h)$ of $a(h^{1/2}x,h^{1/2}\omega,h)$, centered at $(x_0,\omega_0)$ and of degree $M=M_N$, where $M_N$ is defined as at the beginning of the proof of Theorem \ref{teo semi}. As in that proof, we define the quantity $\Gamma_{\varepsilon,r,N}$ and the weight $w_{\varepsilon,r,N,x_0,\omega_0}$ but with the function $|P_{(x_0,\omega_0),M}(x,\omega,h)|$ replaced by $
|P_{(x_0,\omega_0),M}(x,\omega,h)|+b(h^{1/2}x,h^{1/2}\omega,h)$. The key point is that this function still belongs to $A^\alpha_\infty(\bR^{2d})$ for a constant $\alpha\in(0,1)$ depending only on $M=M_N$, i.e. on $N$, by the closure of this class with respect to sums, scaling and multiplication by a positive constant (the above constant $C$ appearing in the expression of $b$, in this case). 

Hence, now
\begin{align*}
&\Gamma_{\varepsilon,r,N}:=\varepsilon\\
&+\inf_{(x_0,\omega_0,t_0)\in\mathbb{H}^d}\fint_{B((x_0,\omega_0,t_0),r)} \big(|P_{(x_0,\omega_0),M}(x,\omega,h)|+b(h^{1/2}x,h^{1/2}\omega,h)\big)\, dx\, d\omega\, dt
\end{align*}
and 
\[
w_{\varepsilon,r,N,x_0,\omega_0}(x,\omega,t):= r^{-2} \Gamma_{\varepsilon,r,N}^{-1} \big(\varepsilon+|P_{(x_0,\omega_0),M}(x,\omega,h)|+b(h^{1/2}x,h^{1/2}\omega,h)\big).
\] 
With these adjustments one can follow the same argument as in the proof of Theorem \ref{teo semi}. 
Only the proof of the inequality analogous to \eqref{eq ult}, namely
\begin{equation}\label{eq ult2}
\Gamma_{\varepsilon,r,N}\geq c_{r,N}(\boldsymbol{\lambda}'(a+b,h)-C''_N h^N), \quad r\in (0,1],\ h\in(0,1].
\end{equation} 
 requires some comments. In this case, we observe that, by Lemma \ref{lem equiv} and Corollary \ref{pro loc doubl}, 
\begin{align*}
&\fint_{B((x_0,\omega_0,t_0),r)} \big(|P_{(x_0,\omega_0),M}(x,\omega,h)|+b(h^{1/2}x,h^{1/2}\omega,h)\big)\, dx\, d\omega\, dt\\
&\geq c_{r,N}\fint_{B((x_0,\omega_0),r)} \big(|P_{(x_0,\omega_0),M}(x,\omega,h)|+b(h^{1/2}x,h^{1/2}\omega,h)\big)\, dx\, d\omega\\
&\geq c'_{r,N}\fint_{B((x_0,\omega_0),1)} |P_{(x_0,\omega_0),M}(x,\omega,h)|\,dx\, d\omega\\
&\qquad\qquad\qquad\qquad\qquad\qquad+c'_{r,N}\fint_{B((x_0,\omega_0),1)} b(h^{1/2}x,h^{1/2}\omega,h)\, dx\, d\omega.
\end{align*}
The first integral can be estimated from below as in the proof of \eqref{eq ult}. For the second integral we observe that, for some constant $c_d$ depending only on $d$,
\[
\fint_{B((x_0,\omega_0),1)} b(h^{1/2}x,h^{1/2}\omega,h)\, dx\, d\omega\geq c_d \sup_{(x,\omega)\in B((x_0,\omega_0),1)} b(h^{1/2}x,h^{1/2}\omega,h),
\]
as one sees from the expression of $b$, using Peetre's inequality: for $z,w\in B((x_0,\omega_0),1)$,
\[
(1+|h^{1/2}z|)^{-1}\geq c'_d (1+|h^{1/2}w|)^{-1} (1+h^{1/2}|z-w|)^{-1}\geq c''_d (1+|h^{1/2}w|)^{-1}. 
\]
Hence one arrives at 
\begin{align*}
&\fint_{B((x_0,\omega_0,t_0),r)} \big(|P_{(x_0,\omega_0),M}(x,\omega,h)|+b(h^{1/2}x,h^{1/2}\omega,h)\big)\, dx\, d\omega\, dt\\
 &\geq c'''_{r,N}\Big(\sup_{(x,\omega)\in B((x_0,\omega_0),1)} (a(h^{1/2}x,h^{1/2}\omega,h)-C''_N h^N)\\&\qquad\qquad\qquad\qquad\qquad\qquad\qquad+\sup_{(x,\omega)\in B((x_0,\omega_0),1)} b(h^{1/2}x,h^{1/2}\omega,h)\Big)\\
 &\geq c'''_{r,N} \Big(\sup_{(x,\omega)\in B((x_0,\omega_0),1)} \big(a(h^{1/2}x,h^{1/2}\omega,h)+b(h^{1/2}x,h^{1/2}\omega,h)\big)-C''_N h^N\Big).
 \end{align*}
Taking the infimum over $(x_0,\omega_0,t_0)\in\mathbb{H}^d$ yields the inequality in \eqref{eq ult2} and therefore concludes the proof of the lower bound \eqref{eq 24}.
 
The second inequality in \eqref{eq 26una} follows from \eqref{eq confron}. 
\end{proof}
\begin{remark}
The lower bound in \eqref{eq ulower} is similar to that in \cite{fefferman_phong}, cf. \eqref{eq feff}, but there are also some obvious differences. Indeed, the privileged scale $h^{1/2}$ is implicit in the definition \eqref{eq anti-w} of the anti-Wick quantization (because of the coherent states $\varphi^h_{(x,\omega)}$), so that only Euclidean balls at that scale enter in the formula for $\boldsymbol{\lambda}'(a,h)$ in \eqref{eq lambdaprime}. Moreover, the better properties of the Anti-Wick quantization (compared with the Weyl quantization) with respect to positivity issues allow us to deal with symbols of arbitrary order. 
\end{remark}
\section*{Appendix}
In this appendix we comment on a characterization, recently proved in \cite{bruno_calzi}, of Schr\"odinger operators with purely discrete spectrum associated with a left-invariant subLaplacian on a noncompact connected Lie group, with a potential in $A_{\infty,{\rm loc}}$.
  
Let $G$ be a noncompact connected Lie group. Consider a left-invariant subRiemannian structure induced by the linearly independent vector fields $X_1,\ldots,X_\ell$ satisfying H\"ormander's condition. Let ${\rm d}_C$ be the corresponding Carnot–Carath\'eodory distance, as in Section \ref{sec prel}.  Let $\lambda$ be a left Haar measure of $G$ and $V\in A_{\infty,{\rm loc}}(G,{\rm d}_C,\lambda)$ (cf. Definition \ref{def muck}). Consider the vector space 
\[
\mathcal{D}'_V:=\{u\in L^2(G,\lambda): |\nabla u|\in L^2(G,\lambda),\ \sqrt{V}u\in L^2(G,\lambda)\}
\] 
where $\nabla u=(X_1u,...,X_\ell u)$ is the horizontal gradient, understood in the sense of distributions. Then 
\[
Q'_V(u):= \int_{G} (|\nabla u|^2+V|u|^2) d\lambda,\qquad u\in\mathcal{D}'_V
\]
is a densely defined nonnegative closed quadratic form, and therefore defines a self-adjoint operator on $L^2(G,\lambda)$ that we will denote by $\mathcal{L}_V$ (cf. \cite{bruno_calzi}). 

Now, it was proved in \cite[Theorem 6.6]{bruno_calzi} that if $V\in A_{\infty,{\rm loc}}(G,{\rm d}_C,\lambda)$ and if the absolutely continuous measure that has density $V$ with respect to $\lambda$ is locally doubling, in the sense that there exist $C, R'>0$ such that
\[
\int_{2B} V\, d\lambda\leq C \int_{B} V\, d\lambda
\]
for every ball $B$ of radius $r\in (0,R']$, then $\mathcal{L}_V$ has purely discrete spectrum if and only if 
\begin{equation}\label{eq bru}
\lim_{x\to\infty}\int_{B(x,R)}V\, d\lambda=+\infty
\end{equation}
for some (equivalently for all) $R>0$. Hence, combining this result with Corollary \ref{pro loc doubl}  we obtain the following characterization.
\begin{theorem}
Let $V\in A_{\infty,{\rm loc}}(G,{\rm d}_C,\lambda)$. The operator $\mathcal{L}_V$ has purely discrete spectrum if and only if \eqref{eq bru} holds for some (equivalently for all) $R>0$.  
\end{theorem}

\section*{Acknowledgments} We would like to thank Tommaso Bruno, Mattia Calzi, Gian Maria Dall'Ara and Eugenia Malinnikova for interesting comments on a preliminary version of the manuscript. We also thank the anonymous Referee for valuable comments. 

The author is a Fellow of the \textit{Accademia delle Scienze di Torino}.  


\end{document}